\documentclass[a4paper,10pt]{article}
\usepackage{tikz}
\usetikzlibrary{matrix}
\usepackage{xcolor}
\usepackage[]{algorithm2e}
\usepackage{bbm}
\usepackage{float}
\usepackage{framed}
\usepackage{amsmath,amssymb,amsfonts,amsthm,esint,dsfont,color}

\definecolor{verde}{rgb}{0.2, 0.7, 0.3}

\newcommand{\ufdPG}{^{\mathrm{fdPG}}}
\newcommand{\uP}{^{\mathrm{P}}}
\newcommand{\uPG}{^{\mathrm{PG}}}
\newcommand{\uPGp}{^{\mathrm{PG,(p)}}}
\newcommand{\uPGpp}{^{\mathrm{PG,(p')}}}
\newcommand{\uPq}{^{\mathrm{P,(q)}}}
\newcommand{\uun}{^{\mathrm{unprec}}}

\newcommand{\matB}{{\mathbf{B}}}
\newcommand{\vecu}{{\mathbf{u}}}
\newcommand{\vecg}{{\mathbf{g}}}

\newtheorem{theorem}{Theorem}[section]

\newtheorem{lemma}[theorem]{Lemma}

\newtheorem{example}[theorem]{Example}
\newtheorem{remark}[theorem]{Remark}

\begin{document}
\title{Low-rank approximation of linear parabolic equations by space-time tensor Galerkin methods}
\author{Thomas Boiveau\thanks{Universit\'e Paris-Est, CERMICS (ENPC), 77455 Marne-la-Vall\'ee 2, France and Inria Paris, 75589 Paris, France. Part of this research was carried out when the first and third authors participated in the Hausdorff Trimester Program  ``Multiscale Problems: Algorithms, Numerical Analysis and Computation''}
\and Virginie Ehrlacher\footnotemark[1]
\and Alexandre Ern\footnotemark[1]
\and Anthony Nouy\thanks{GeM - UMR CNRS 6183, \'Ecole Centrale de Nantes, Universit\'e de Nantes, 1 rue de la No\"e, 44321 Nantes, France.}}
\date{}
\maketitle
\begin{abstract}
We devise a space-time tensor method for the low-rank approximation of linear parabolic evolution equations. The proposed method is a stable Galerkin method, uniformly in the discretization parameters, based on a Minimal Residual formulation of the evolution problem in Hilbert--Bochner spaces. The discrete solution is sought in a linear trial space composed of tensors of discrete functions in space and in time and is characterized as the unique minimizer of a discrete functional where the dual norm of the residual is evaluated in a space semi-discrete test space. The resulting global space-time linear system is solved iteratively by a greedy algorithm. Numerical results are presented to illustrate the performance of the proposed method on test cases including non-selfadjoint and time-dependent differential operators in space. The results are also compared to those obtained using a fully discrete Petrov--Galerkin setting to evaluate the dual residual norm.
\end{abstract}
%
%

\medskip\noindent\textbf{AMS.}
{65M12, 65M22, 35K20}
%
\section{Introduction}

The goal of this work is to devise a space-time tensor method for the low-rank approximation of the solution to linear parabolic evolution equations. The method we propose has two salient features. First, it is a stable Galerkin method, uniformly with respect to the discretization parameters in space and in time, leading to quasi-optimal error estimates in the natural norms of the problem as specified below. Second, the method is global in time (and in space), and the approximate solution is iteratively constructed by solving alternatively global problems in space and in time. More precisely, at iteration $m\in\mathbb{N}^*$, the space-time approximate solution $u^m(x,t)$ is of the form
\begin{equation} \label{eq:greedy1}
u^m(x,t) = \sum_{1\le n\le m} v^n(x)s^n(t),
\end{equation} 
where $v^n$ and $s^n$ are members of some finite-dimensional spaces $V_h$ and $S_k$ composed of space and time functions having dimension $N_h$ and $N_k$, respectively. We say that $u^m$ is an approximation of rank $m$ of the exact solution, consisting of a summation of rank-one terms. We employ a greedy rank-one algorithm for computing the sequence of approximations. Specifically, the approximate solution $u^m$ is constructed iteratively, i.e., once $u^{m-1}$ for $m\ge1$ is known (we set conventionally $u^0=0$), the functions $v^m$ and $s^m$ are computed by solving successively a global problem in space and in time having size $N_h$ and $N_k$, respectively, and which is defined by minimizing some quadratic convex functional. The present method has the potential to be computationally effective if the exact solution can be approximated accurately by low-rank space-time tensors. In this case, space-time compression is meaningful and full parallelism in time can be exploited by the global time solves leading to an overall computational cost of the order of $m(N_h+N_k)$, whereas traditional time-stepping methods are expected to exhibit a cost of the order of $N_h\times N_k$. Thus, computational benefits are expected whenever $m\ll \min(N_h,N_k)$.
We notice that in the literature, approximate solutions of the form~\eqref{eq:greedy1} are typically obtained  by a Proper Generalized Decomposition (PGD)~\cite{Le-Bris_2009_a, Nouy_2009_a, Falco_2011_a, Chinesta_2014_a}, which is a greedy algorithm~\cite{Temlyakov_2008_a}. 
In the context of parabolic problems, the PGD strategy has been first introduced within the LATIN method in~\cite{Ladeveze_2012_a}. Theoretical convergence results have been obtained in different contexts, see, e.g., \cite{Chinesta_2014_a, Le-Bris_2009_a, Cances_2011_a,Falco_2012_a, cances2014greedy}.
We leave the question of adaptivity to future work, i.e., the discrete spaces $V_h$ and $S_k$ are fixed a priori in what follows. Possible applications of the present method can be envisaged in optimal control problems constrained by parabolic evolution equations (see, e.g., \cite{GunKu:11}) and in parabolic evolution equations with random input data (see, e.g., \cite{HoaSc:13}); in both cases indeed, global space-time approaches are important. We also mention the dynamical low-rank integrators from, e.g., \cite{KocLu:07,LubOs:14,KiLuW:16} which can be used for parabolic evolution equations with the rank referring to the space variables.

Our starting point is the well-posed formulation of the parabolic evolution equation at hand in the setting of space-time Hilbert--Bochner spaces. More precisely, following Lions and Magenes \cite[p.~234]{LioMa:72}, the trial space is $X = L^2(I;V)\cap H^1(I; V')$ and the test space is $Z=L^2(I;V)\times L$, where $I$ is the (non-empty, bounded) time interval and the separable real Hilbert spaces $(V,L,V')$ form a Gelfand triple, i.e., $V\hookrightarrow L\equiv L'\hookrightarrow V'$ with densely defined embeddings. The prototypical example is the heat equation for which $V=H^1_0(\Omega)$, $L=L^2(\Omega)$, $V'=H^{-1}(\Omega)$, where $\Omega$ is a bounded, Lipschitz, open subset of $\mathbb{R}^d$, $d\ge1$. More generally, we consider time-dependent linear (possibly non-selfadjoint) operators $A(t) : V \rightarrow V'$ that are bounded and coercive, pointwise in time (a.e.). One important assumption for the present method to remain computationally effective is that the space-time operator and source term in the parabolic evolution equation admit a separated representation in space and in time with a relatively low rank, see Eq.~\eqref{eq:separate} below. This assumption is met in practice for a wide range of problems coming from the engineering and applied sciences. In several situations, it is also possible to devise such low-rank separated representations with good accuracy using the Empirical Interpolation Method (EIM) \cite{BaMNP:04}. The above well-posed space-time formulation allows one to view the parabolic evolution problem as a Minimal Residual (MinRes) formulation, where the exact solution is the unique minimizer over the trial space $X$ of the (square of the) dual norm of the residual in $Z'$. The present approximation method is formulated as a Galerkin method for the MinRes formulation. More precisely, we look for the minimizer over a finite-dimensional subspace $X_{hk}\subset X$ of the (square of the) dual norm of the residual measured with respect to a space semi-discrete subspace $Z_h\subset Z$. Here, $X_{hk}$ is a linear space generated using space-time tensor-products of elements of the finite-dimensional subspaces $V_h\subset V$ and $S_k\subset H^1(I)$ considered above, whereas $Z_h = (V_h\otimes L^2(I))\times V_h$, i.e., the time variable is not discretized in $Z_h$. 

Let us put our work in perspective with the literature. The subject of numerical methods to approximate parabolic evolution equations is extremely rich. One important class of methods are the traditional time-stepping schemes which approximate the solution on a succession of time sub-intervals by marching along the positive time direction, see, e.g., \cite{Thomee:06} for an overview. In the context of parabolic evolution equations, implicit schemes are often preferred to circumvent the classical CFL restriction on explicit schemes, which is of the form $\delta_k\lesssim \delta_h^2$ (where $\delta_k\sim N_k^{-1}$ is the time-step and $\delta_h\sim N_h^{-1/d}$ is the space-step), but at the expense of having to solve a large linear system of equations at each time-step. The cost of having to solve these systems sequentially has motivated the devising of parareal methods
\cite{LiMaT:01,GanVa:07} based on iterative corrections at all time sub-intervals simultaneously using the global space-time discrete system associated with time-stepping methods. This global space-time discrete system has also been central in the devising of space-time domain decomposition methods based on  waveform relaxation \cite{JanVa:96,GanZh:02,GilKe:02}.
In contrast to the above approaches which do not make a direct use of the well-posed functional setting in space-time Hilbert--Bochner spaces, a space-time adaptive wavelet method for parabolic evolution problems was proposed and analyzed in \cite{Schwab_2009_a}, involving a rather elaborate construction of the wavelet bases. Simpler hierarchical wavelet-type tensor bases on a space-time sparse grid were also considered in \cite{GriOe:07} within a heuristic space-time adaptive algorithm, but without offering guaranteed a-priori stability, uniformly with respect to the discretization parameters. We also mention the recent work \cite{NeuSme:18} where the above functional setting for parabolic evolution problems is used to devise preconditioned time-parallel iterative solvers. Furthermore, PGD approximations based on a discrete MinRes formulation measured in the space-time Euclidean norm of the components in a basis of $X_{hk}$ have been devised and evaluated numerically in \cite{Nouy_2010_a}, obtaining promising results on various model parabolic evolution problems.

More recently, space-time Petrov--Galerkin discretizations of parabolic evolution equations were proposed and analyzed in the PhD Thesis \cite{Andreev:12} and in the related papers \cite{Andreev_13,Andreev_2014_a}. Therein, the same MinRes formulation is considered as in the present work at the continuous level, and the approximate solution is typically sought in the same space-time discrete space. There are, however, several salient differences between \cite{Andreev_13,Andreev_2014_a} and the present work. First, in \cite{Andreev_13,Andreev_2014_a}, the dual residual norm of the discrete solution is measured with respect to a \emph{fully discrete} space-time test space, leading to a Petrov--Galerkin formulation, whereas we consider a \emph{space semi-discrete} test space, leading to a standard Galerkin formulation. The difference is that a careful design of the test space is necessary in the Petrov--Galerkin setting. Precise results in this direction have been obtained in \cite{Andreev_13}. Let us for instance notice that, for a constant-in-time and selfadjoint differential operator in space (e.g., for the heat equation), the Crank--Nicolson scheme (obtained using continuous piecewise affine time-functions in the trial space and piecewise constant time-functions in the test space) is only conditionally stable, with a stability constant degenerating with the parabolic CFL condition, whereas unconditional stability is achieved by further refining the time mesh used to build the discrete test space as shown in \cite[Sec.~5.2.3]{Andreev:12}. In contrast with this, the present formulation automatically inherits uniform stability with respect to the time-step size. In particular, Lemma~\ref{lem:Bh_properties} below shows well-posedness and Lemma~\ref{lem:err_est_Gal} a quasi-optimal error bound with a constant independent on the time discretization. Nonetheless, the condition number of the discrete matrices should behave as $O(N_k^{-2})$, so that, as usual, the precision can be limited in practice by round-off errors.  
The second difference lies in the way the discrete system of linear equations is solved iteratively: we use a greedy algorithm to build a sequence of approximate solutions having the form~\eqref{eq:greedy1}, whereas (a generalization of) the LSQR algorithm \cite{PaiSa:82} is used in \cite{Andreev_2014_a}. Third, we report numerical results on a larger set of model parabolic evolution problems, including in particular non-selfadjoint operators of advection-diffusion-type at moderate P\'eclet numbers. Finally, let us mention, as already observed in \cite{Andreev:12,Andreev_13,Andreev_2014_a}, that the inner product with which we equip the discrete test space $Z_h$ plays the role of a preconditioner in the discrete system of linear equations resulting from the discrete MinRes formulation. Equipping $Z_h$ with the natural norm leads to the appearance of the inverse of the stiffness matrix in space. Herein, we explore numerically the effect of relaxing the use of this preconditioner by simply equipping the space-part of $Z_h$ with the Euclidean norm of the components in a basis of $V_h$. This corresponds to the approach considered in \cite{Nouy_2010_a}.

In Section \ref{sec:continuous_setting}, we specify the functional setting for parabolic evolution equations and the MinRes formulation. 
This formulation is the basis for the discrete MinRes Galerkin formulation devised in Section~\ref{sec:discrete}, where one key idea is the use of a space semi-discrete test space to measure the dual norm of the residual.
In Section \ref{sec:low_rank_approximation}, we present the greedy algorithm we consider to obtain a low-rank approximation of the discrete solution. 
In Section \ref{sec:other}, we present two other discrete MinRes formulations, one using a fully discrete Petrov--Galerkin setting as in \cite{Andreev:12,Andreev_13,Andreev_2014_a} and one using the same space semi-discrete setting as in Section~\ref{sec:discrete} but equipping the test space with the above-mentioned Euclidean norm. These additional formulations are introduced for the purpose of performing numerical comparisons. 
Numerical results on various test cases are discussed in Section \ref{sec:numerical_results}.
Finally, conclusions are drawn in Section~\ref{sec:conclusion}.

\section{Minimal Residual formulation of parabolic evolution equations}\label{sec:continuous_setting}

In this section, we present the MinRes formulation of parabolic equations, which is at the heart of the tensor approximation method we propose later on.

\subsection{Parabolic equations}

The functional setting for parabolic equations is well understood
(see, e.g., the textbooks by Lions and Magenes \cite[p.~234]{LioMa:72}, 
Dautray and Lions \cite[p.~513]{DauLi:92}, and Wloka \cite[p.~376]{Wloka:87}).
Let 
\begin{equation} \label{eq:Gelfand}
V\hookrightarrow L\equiv L' \hookrightarrow V'
\end{equation}
be a Gelfand triple where $V$ and $L$ are separable real Hilbert spaces respectively equipped with inner products $\langle\cdot,\cdot \rangle_{V}$ and $\langle\cdot,\cdot \rangle_{L}$, 
with associated norms $\|\cdot\|_V$ and $\|\cdot\|_{L}$. The symbol $\hookrightarrow$ represents a densely defined and continuous embedding. Let $T>0$ be the time horizon and let $I:=(0,T)$ 
be the time interval. Let $A : I \rightarrow \mathcal{L}(V,V')$ be a strongly measurable time-function with values in the Hilbert space of bounded linear operators from $V$ to $V'$. We assume 
that the following boundedness and coercivity properties hold true: 
there exist $0<\alpha \le M<+\infty$ such that a.e.~$t\in I$,
\begin{subequations} \label{eq:alpha_M} \begin{alignat}{2} \label{eq:M}
&\|A(t)v\|_{V'}\leq M\|v\|_{V},&\quad\quad &\forall v\in V, \\
\label{eq:alpha}
&\langle A(t)v,v \rangle_{V',V}\geq \alpha \| v \|_{V}^2,&\quad\quad &\forall v\in V.
\end{alignat}\end{subequations}
We do not require $A(t)$ to be selfadjoint. 

Let us define the Hilbert--Bochner spaces
\begin{equation}
X:=L^2(I;V)\cap H^1(I; V'),\quad\quad Y:=L^2(I;V),\quad\quad Z:=Y\times L.
\end{equation}
Since $X \hookrightarrow \mathcal{C}^0(\overline{I}; L)$ with $\overline{I}=[0,T]$, the value at any time $t\in \overline{I}$ of any function $x\in X$ is well-defined as an element of $L$. In particular, we denote $x(0)\in L$ the initial value of $x$ at the time $t=0$.
The spaces $X$ and $Z$ are equipped with the norms
\begin{subequations}\begin{align}
\label{norm_X_definition}
\| x \|_{X}^2&:=\| x\|_{L^2(I;V)}^2+M^{-2}\| \partial_t x \|_{L^2(I;V')}^2+\alpha^{-1}\|x(T)\|_L^2,\qquad\forall x \in X,\\
\|z\|_{Z}^2&:= \|y\|_{L^2(I;V)}^2+\alpha^{-1}\|g\|_{L}^2,\qquad\forall z=(y,g) \in Y\times L = Z, 
\end{align}\end{subequations}
where the various scaling factors are introduced to be dimensionally consistent.

Let $f\in Y'=L^2(I;V')$ and let $u_0\in L$. We consider the following parabolic problem: find $u\in X$ such that 
\begin{equation}
\label{evolution_problem}
\left\{
\begin{alignedat}{2}
&\partial_t u(t)+A(t)u(t)=f(t),&\qquad&\text{in $V'$, a.e.~$t\in I$},\\
&u(0) = u_0,&\qquad&\text{in $L$}.
\end{alignedat}\right.
\end{equation}
For the present space-time tensor methods to be computationally effective, we assume that the operator $A$ and the source term $f$ have the following separated form:
\begin{equation}\label{eq:separate}
A(t) = \sum_{1\le p\le P} \mu^{(p)}(t) A^{(p)}\in \mathcal{L}(V,V'), \qquad f(t)  = \sum_{1\le q\le Q} \lambda^{(q)}(t)f^{(q)} \in V', 
\end{equation}
for some positive integers $P,Q$ taking moderate values, where $\mu^{(p)} \in L^{\infty}(I)$, $A^{(p)} \in \mathcal{L}(V,V')$ for all $1\leq p \leq P$, and $\lambda^{(q)}\in L^2(I)$, $f^{(q)}\in V'$ for all $1\leq q \leq Q$. A similar decomposition is considered, e.g., in \cite{MaScTo:18} for the space-time isogeometric discretization of parabolic problems with varying coefficients.

\begin{example}[Heat equation]
Let $\Omega$ be a Lipschitz domain in $\mathbb{R}^d$, $V = H^1_0(\Omega)$, $L = L^2(\Omega)$, and $V' = H^{-1}(\Omega)$. Let $\mu: I\rightarrow\mathbb{R}$ be a measurable function bounded from above and from below away from zero uniformly on $I$. Then, the time-dependent family of operators $(A(t))_{t\in I}$ defined such that $A(t):=-\mu(t)\Delta$, for a.e.~$t\in I$, where $\Delta \in \mathcal L(V;V')$ is the Laplacian operator, satisfies the assumptions~\eqref{eq:alpha_M} and~\eqref{eq:separate}. Whenever $\mu(t)\equiv1$, the family of operators is time-independent, and one recovers the prototypical example of the heat equation.
\end{example}

\subsection{Well-posedness and minimal residual formulation}

It is convenient to introduce the operator $\mathcal{A} : X \rightarrow Y'$ so that 
\begin{equation}
(\mathcal{A}x)(t) = \partial_tx(t)+A(t)x(t) \in V', \quad\quad \text{a.e. $t\in I$}.
\end{equation}
Problem~\eqref{evolution_problem} can be equivalently rewritten using the operator $\mathcal{B}:X\rightarrow Z' = Y' \times L$ such that
\begin{equation}
\mathcal{B}x = (\mathcal{A}x, x(0)) \in Z', \qquad \forall x\in X. 
\end{equation}
Then, an equivalent reformulation of problem \eqref{evolution_problem} reads as follows: find $u\in X$ such that
\begin{equation}
\label{continuous_system}
\mathcal{B}u= (f, u_0) \quad\text{in $Z'$}.
\end{equation}
It is well-known that the operator $\mathcal{B}$ is bounded, i.e., $\mathcal{B}\in \mathcal{L}(X,Z')$, and satisfies 
the following two properties:
\begin{subequations}\begin{align}
&\exists \beta>0 \quad \text{s.t.}\quad \inf_{x\in X}\sup_{z\in Z} \frac{\langle \mathcal{B}x,z\rangle_{Z',Z}}{\|x\|_X\|z\|_Z} \ge \beta, \label{eq:infsup}\\
&\forall z\in Z, \quad (\langle \mathcal{B}x,z\rangle_{Z',Z}=0, \; \forall x\in X) \; \Longrightarrow \; (z=0),
\end{align}\end{subequations} 
where it is implicitly understood that nonzero arguments are considered in the inf-sup condition, and where, for all $(y,g)\in Y\times L=Z$,
\begin{align}
\langle \mathcal{B}x,(y,g)\rangle_{Z',Z}&= \langle\mathcal{A}x,y\rangle_{Y',Y} + \langle x(0),g\rangle_L \nonumber \\
&= \int_I \langle \partial_tx(t)+A(t)x(t),y(t)\rangle_{V',V}dt + \langle x(0),g\rangle_L. \label{eq:def_calB_op}
\end{align}
Therefore, owing to the Banach--Ne\v{c}as--Babu\v{s}ka Theorem (see, e.g., \cite[Thm.~2.6]{Ern_2004_a}), $\mathcal{B}$ is an isomorphism. 
The proof of the well-posedness of parabolic problems by means of inf-sup arguments can be found in \cite[Thm.~6.6]{Ern_2004_a} using a strongly enforced initial condition; a systematic treatment can be found more recently in \cite{Tantardini_2016_a}. Since the operator norm of $\mathcal{B}$ and the inf-sup constant $\beta$ in \eqref{eq:infsup} play an important role in what follows, we provide respectively an upper bound and a lower bound for these two constants. For a.e.~$t\in I$, the inverse adjoint operator $A(t)^{-\mathrm{T}}$ is well-defined in $\mathcal{L}(V';V)$, and we have $M^{-1}\|\varphi\|_{V'}\le \|A(t)^{-\mathrm{T}}\varphi\|_V \le \alpha^{-1}\|\varphi\|_{V'}$ and $\langle \varphi,A(t)^{-\mathrm{T}}\varphi\rangle_{V',V} \ge \frac{\alpha}{M^2}\|\varphi\|_{V'}^2$, for all $\varphi\in V'$.
\begin{lemma}[Boundedness] \label{lem:continuity}
The norm of the operator $\mathcal{B}: X \to Z'$  is such that  $$\Vert \mathcal{B} \Vert_{\mathcal{L}(X;Z')} := \sup_{x\in X}\sup_{z\in Z} \frac{\langle \mathcal{B}x,z\rangle_{Z',Z}}{\|x\|_X\|z\|_Z}  \le \sqrt{3} M.$$
\end{lemma}
\begin{proof}
For all  $x\in X$, $z = (y,g)\in Z = Y\times L$, we have 
\begin{align*}
\langle \mathcal{B}x,z\rangle_{Z',Z} &= \langle\mathcal{A}x,y\rangle_{Y',Y} + \langle x(0),g\rangle_L \le 
\Vert \mathcal{A} x \Vert_{Y'} \Vert y \Vert_Y + \Vert x(0) \Vert_L \Vert g \Vert_L\\
&\le \left(\Vert \mathcal{A} x \Vert_{Y'}^2 + \alpha \Vert x(0) \Vert_L^2 \right)^{1/2}  \left(\Vert y \Vert_{Y}^2 + \alpha^{-1} \Vert g \Vert_L^2 \right)^{1/2} 
\le \sqrt{3} M \Vert x \Vert_X\Vert z \Vert_Z,
\end{align*}
since $\Vert \mathcal{A} x \Vert_{Y'}^2\le 2M^2(M^{-2}\Vert \partial_t x \Vert_{Y'}^2 + \Vert x \Vert_{Y}^2)\le 2M^2\|x\|_X^2$ and $\alpha \Vert x(0) \Vert_L^2 \le
\alpha \Vert x(T) \Vert_L^2 +2\alpha\Vert \partial_t x \Vert_{Y'}\Vert x \Vert_{Y}
\le \alpha \Vert x(T) \Vert_L^2 + \alpha M(M^{-2}\Vert \partial_t x \Vert_{Y'}^2 + \Vert x \Vert_{Y}^2)\le M^2\|x\|_X^2$ (recall that $\alpha\le M$).
\end{proof}

\begin{lemma}[Inf-sup constant] \label{lem:infsup}
\eqref{eq:infsup} holds true with the inf-sup constant $\beta\ge \frac{\alpha^2}{M}(1+\kappa^2)^{-\frac12}$, where $\kappa:=\mathop{\mbox{\em ess~sup}}_{t\in I}\frac12\|A(t)A(t)^{-\mathrm{T}}-I\|_{\mathcal{L}(V';V')}$.
\end{lemma}
\begin{proof}
For completeness, we briefly outline the proof which is similar to that of \cite[Prop.~3.1]{Tantardini_2016_a}, but with a different scaling for the norms. Let $x\in X$ and let us take
\[
z = (A(t)^{-\mathrm{T}}\partial_tx(t)+ x(t),x(0)) \in Z.
\]
Then, using the coercivity of $A(t)$ and of $A(t)^{-1}$, we have
\begin{align*}
\langle \mathcal{B}x,z\rangle_{Z',Z} &= \int_I \langle \partial_tx(t)+A(t)x(t),A(t)^{-\mathrm{T}}\partial_tx(t)+ x(t)\rangle_{V',V}dt + \|x(0)\|_L^2 \\
&\ge \alpha M^{-2} \|\partial_tx\|_{L^2(I;V')}^2 + \alpha\|x\|_{L^2(I;V)}^2 + \|x(T)\|_L^2 \ge \alpha \|x\|_X^2.
\end{align*}
Moreover, using again the coercivity of $A(t)$ and the boundedness of $A(t)$ and $A(t)^{-\mathrm{T}}$, we have
\begin{align*}
&\|z\|_Z^2 = \int_I \|A(t)^{-\mathrm{T}}\partial_tx(t)+ x(t)\|_V^2dt + \alpha^{-1}\|x(0)\|_L^2 \\ 
\le{}& \alpha^{-1} \int_I \langle A(t)\big(A(t)^{-\mathrm{T}} \partial_t x(t)+ x(t)\big),A(t)^{-\mathrm{T}}\partial_tx(t)+ x(t)\rangle_{V',V}dt + \alpha^{-1}\|x(0)\|_L^2 \\
\le{}& M\alpha^{-1}\|x\|_{L^2(I;V)}^2+M^2\alpha^{-2}\|\partial_t x\|_{L^2(I;V')}^2 + \alpha^{-1}\|x(T)\|_L^2 \\&+ 2\kappa \alpha^{-1} \|\partial_t x\|_{L^2(I;V')}\|x\|_{L^2(I;V)} \\
\le{}& (1+\kappa^2)M^2\alpha^{-2} \|x\|_X^2,
\end{align*}
and the conclusion is straightforward.
\end{proof}

\begin{remark}[Heat equation]
Sharp estimates of the inf-sup constant $\beta$ for the heat equation (with $\mu(t) \equiv 1$ on $I$ so that $\alpha=M=1$ and $\kappa=0$) can be found in \cite{UrbPa:12,ErnSV:17} using the above norms.
\end{remark}

The solution to the parabolic equation (\ref{evolution_problem}) is the global minimizer of the residual-based quadratic functional $\mathcal{E}:X\rightarrow \mathbb{R}$, defined such that 
\begin{equation} \label{eq:def_calE}
\mathcal{E}(x):=\frac12 \|\mathcal{B}x-(f,u_0)\|_{Z'}^2, \qquad \forall x\in X,
\end{equation}
where we equip the space $Z'=Y'\times L$ with the norm
\begin{equation}
\|(\phi,g)\|_{Z'}^2 := \|\phi\|_{L^2(I;V')}^2 + \alpha\|g\|_L^2.
\end{equation}
Since the functional $\mathcal{E}$ is strongly convex on $X$ with parameter $\beta^2>0$ owing to the inf-sup condition~\eqref{eq:infsup}, $\mathcal{E}$ admits a unique global minimizer in $X$, and since the operator $\mathcal{B}$ is surjective, the minimum value of $\mathcal{E}$ on $X$ is zero. 
In other words, the unique solution to~\eqref{evolution_problem}
can be equivalently characterized as follows:
\begin{equation}\label{minimisation_problem}
u=\underset{x\in X}{\textup{argmin}}\; \frac12 \left(\|\mathcal{A}x-f\|_{Y'}^2+\alpha\|x(0)-u_0\|_{L}^2\right).
\end{equation}


\section{Discrete Minimal Residual Galerkin formulation}
\label{sec:discrete}

In this section, we introduce a discrete energy based on a space semi-discrete Galerkin method to approximate the unique minimizer in~\eqref{minimisation_problem}. 
We consider finite-dimensional spaces $V_{h}$ and $S_{k}$ such that
\begin{equation}
V_{h}\subset V,\quad\quad S_{k}\subset S:= H^1(I),
\end{equation}
and we set $N_h:=\textup{dim}(V_h)$ and $N_k:=\textup{dim}(S_k)$.
Typically, $V_h$ is constructed using $\mathbb{P}_1$ Lagrange finite elements on a space 
mesh of $\Omega$ and $S_k$ is constructed using continuous, piecewise affine functions on a time mesh
of $I$. We are going to seek the discrete minimizer in the tensor-product space 
\begin{equation}
X_{hk}:= V_{h}\otimes S_{k} \subset X,
\end{equation}
which is of dimension $\textup{dim}(X_{hk})=N_h\times N_k$.

\subsection{Space semi-discrete Galerkin approximation}
Let us set 
\begin{equation}
X_h = V_h \otimes H^1(I), \quad Y_{h}:= V_h\otimes L^2(I) \equiv L^2(I;V_h), \quad Y_h'=V_h'\otimes L^2(I) \equiv L^2(I;V_h').
\end{equation} 
Since $V_h$ is a subspace of $V$ and owing to~\eqref{eq:Gelfand}, we have $Y_h \subset Y$ and $Y' \subset Y_h'$, where the second inclusion follows by restricting the action of linear forms on $V$ to $V_h$. Let us define $f_h\in Y_h'$ s.t. $f_h(t) = f(t)|_{V_h}$ a.e.~$t\in I$. Recalling the separated form~\eqref{eq:separate}, we have 
\begin{equation} \label{eq:def_fh}
f_h(t):=\sum_{1\le q\le Q} \lambda^{(q)}(t) f^{(q)}|_{V_h} \qquad\text{a.e.~$t\in I$}. 
\end{equation}
Let $u_{0h}$ be the $L$-orthogonal projection of $u_0$ onto $V_h$. 
Let $\mathcal{A}_h : X_{h} \rightarrow Y_h'$ be s.t.
\begin{equation} \label{eq:def_Ah}
(\mathcal{A}_{h} x_{h})(t):= (J_{V_h}\otimes \partial_t) x_{h}(t) + A_{h}(t) x_{h}(t) \in V_h', \qquad \text{a.e.~$t\in I$},
\end{equation}
where $J_{V_h}:V_h\rightarrow V_h'$ is the injection resulting from~\eqref{eq:Gelfand}, i.e.,  $\langle J_{V_h}v_h,w_h\rangle_{V_h',V_h} = \langle v_h,w_h\rangle_L$ for all $v_{h},w_{h}\in V_{h}$, and $A_{h}(t):V_{h}\rightarrow V_{h}'$ is the discrete counterpart of $A(t)$ s.t.~$\langle A_h(t)v_h,w_h \rangle_{V_h',V_h}=\langle A(t)v_h,w_h \rangle_{V',V}$. 
Let us introduce the space semi-discrete space $Z_h=Y_h\times L_h$ so that $Z_h'=Y_h'\times L_h'$, where $L_h$ coincides with $V_h$ as linear space but is equipped with the norm of $L$ (note that $L_h\subset L\equiv L'\subset L_h'$). Note that $Z_h\subset Z$. Let  $\mathcal{B}_{h} : X_{h}\rightarrow Z_h'$ be the operator 
defined for $x_h\in X_h$ by  
\begin{equation}
\mathcal{B}_h x_h = (\mathcal{A}_h x_h , x_h(0)) \in  Z_h' = Y_h' \times L_h', 
\end{equation}
and such that, for all $(y_h,g_h)\in Y_h\times L_h$, (compare with~\eqref{eq:def_calB_op})
\begin{align}
\langle \mathcal{B}_{h}x_{h},(y_h,g_h)\rangle_{Z_h',Z_h}&= \langle\mathcal{A}_hx_{h},y_h\rangle_{Y_h',Y_h} + \langle x_{h}(0),g_h\rangle_{L}= \langle \mathcal{B} x_{h},(y_h,g_h)\rangle_{Z',Z}.
 \label{eq:def_calB_op_h}
\end{align}
The space semi-discrete formulation is as follows: find $u_{h}\in X_{h}$ such that
\begin{equation}
\mathcal{B}_h u_h = (f_h,u_{0h}) \quad \text{in } Z_h'.\label{semi-discrete-equation}
\end{equation}
The well-posedness of this formulation is ensured by the following lemma,
where the subspace $X_h\subset X$ is equipped with the norm
$\|x_h\|_{X_h}^2 := \|x_h\|_{L^2(I;V)}^2+M^{-2}\| \partial_t x_h \|_{L^2(I;V_h')}^2+\alpha^{-1}\|x_h(T)\|_L^2$, and the subspace $Z_h\subset Z$ is equipped with the 
norm of $Z$.
Recall the inf-sup constant $\beta =  \frac{\alpha^2}{M}(1+\kappa^2)^{-\frac12}$ from Lemma~\ref{lem:infsup}.

\begin{lemma}\label{lem:Bh_properties}
The operator $\mathcal{B}_h : X_h\to Z_h'$ satisfies $\Vert \mathcal{B}_h \Vert_{\mathcal{L}(X_h;Z_h')} \le \sqrt{3} M,$
and 
\begin{subequations}\begin{align}
&  \inf_{x_h\in X_h}\sup_{z_h\in Z_h} \frac{\langle \mathcal{B}_hx_h,z_h\rangle_{Z_h',Z_h}}{\|x_h\|_{X_h}\|z_h\|_{Z}} \ge \beta, \label{eq:infsup_Bh}\\
&\forall z_h\in Z_h, \quad (\langle \mathcal{B}_hx_h,z_h\rangle_{Z_h',Z_h}=0, \; \forall x_h\in X_h) \; \Longrightarrow \; (z_h=0).\label{eq:infsup_Bh_bis}
\end{align}\end{subequations} 
\end{lemma}

\begin{proof}
The upper bound on $\Vert \mathcal{B}_h \Vert_{\mathcal{L}(X_h;Z_h')}$ is shown as in the proof of Lemma~\ref{lem:continuity}. 
To prove \eqref{eq:infsup_Bh}, one can use the same arguments as in the proof of Lemma~\ref{lem:infsup}
by picking in the supremizing set $z_{h} = (A_{h}(t)^{-\mathrm{T}}\partial_tx_{h}(t)+ x_{h}(t),x_{h}(0)) \in Y_h\times L_h=Z_h$. Finally, one can prove~\eqref{eq:infsup_Bh_bis} using the following arguments, as in \cite[Thm.~6.6]{Ern_2004_a}. Let $z_h=(y_h,g_h)\in Z_h$. Taking $x_h$ arbitrary in $V_h\otimes C_0^\infty(I)$ shows that $(J_{V_h}\otimes \partial_t)y_h - A_h^{-\mathrm{T}}y_h=0$ in $Y_h'$. Taking next $x_h=tv_h$ with $v_h$ arbitrary in $V_h$ proves that $y_h(T)=0$, and taking $x_h=ty_h$, one concludes that $y_h=0$. Finally, taking $x_h=v_h$ with $v_h$ arbitrary in $V_h$ yields $g_h=0$.
\end{proof}

Lemma~\ref{lem:Bh_properties} implies that $\mathcal{B}_h : X_h\to Z_h'$ is 
an isomorphism such that  
\begin{equation} 
\beta \Vert x_h \Vert_{X_h} \le \Vert \mathcal{B}_h x_h \Vert_{Z_h'} \le \sqrt{3} M  \Vert x_h \Vert_{X_h},\label{Bh-isomorphism}
\end{equation}
for all $x_h \in X_h$. The solution $u_h$ to the equation \eqref{semi-discrete-equation} is the unique minimizer of  
the discrete energy functional $\mathcal{E}_{h}:X_{h}\rightarrow \mathbb{R}$ defined for all $x_{h}\in X_{h}$ by
\begin{equation} \label{eq:def_calE_sdG}
\mathcal{E}_{h}(x_{h}):= \frac12 \Vert \mathcal{B}_h x_h - (f_h,u_{0h}) \Vert_{Z_h'}^2 
= \frac12 \left(\|\mathcal{A}_hx_h-f_h\|_{Y_h'}^2+\alpha\|x_h(0)-u_{0h}\|_{L}^2\right).
\end{equation}
An important property of this discrete energy functional is the strong convexity that is inherited from the continuous setting, uniformly with respect to the  space discretization parameter. More precisely, the functional $\mathcal{E}_h$ is strongly convex on $X_h$ with parameter $\beta^2>0$ owing to the inf-sup condition~\eqref{eq:infsup_Bh}. Since the operator $\mathcal{B}_h$ is surjective, the minimum value of $\mathcal{E}_h$ on $X_h$ is zero and is attained at $u_h$.

\begin{remark}[Norm $\|{\cdot}\|_{X_h}$]
The difference between the $\|{\cdot}\|_X$-norm and the
$\|{\cdot}\|_{X_h}$-norm lies in the use of the
dual norm in $V_h'$ and not in $V'$ to measure the time-derivative. Note that
$\|x_h\|_{X_h}\le \|x_h\|_{X}$, for all $x_h\in X_h$. The reason for this 
difference is that, as shown in 
\cite{Tantardini_2016_a}, the equivalence of the two norms, uniformly 
with respect
to the space discretization, holds true if and only if the  
$L$-orthogonal projection onto $V_h$ is $V$-stable. This uniform stability 
(with $V=H^1_0(\Omega)$ and $L=L^2(\Omega)$) is, in turn,
not known to hold true if general shape-regular meshes are used to build the
finite element space $V_h$; it does hold true if quasi-uniform meshes are used 
(as it is the case in the present numerical experiments). We emphasize that the
use of a discrete dual norm to measure the time-derivative is a general
feature that arises in the quasi-optimality of space semi-discrete
Galerkin methods for parabolic evolution problems
\cite{Tantardini_2016_a}, and
is not specific to the present setting.
\end{remark}

\subsection{Minimal residual Galerkin approximation}

An approximation  $u_{hk} \in X_{hk}$ of  $u_h \in X_h$ is now defined as the unique minimizer of the discrete energy functional $\mathcal{E}_h$ restricted to the subspace $X_{hk}$ of $X_h$, i.e.we look for
\begin{equation}
\label{minimisation_full_discrete}
u_{hk}=\underset{x_{hk}\in X_{hk}}{\textup{argmin}} \mathcal{E}_{h}(x_{hk})  = \underset{x_{hk}\in X_{hk}}{\textup{argmin}}  \frac12 
\left(\|\mathcal{A}_hx_{hk}-f_h\|_{Y_h'}^2+\alpha\|x_{hk}(0)-u_{0h}\|_{L}^2\right).
\end{equation}
We emphasize the use of the space semi-discrete test space $Y_h$ to measure the dual norm of the residual.
We obtain the following quasi-optimal error estimate.

\begin{lemma}[Error estimate] \label{lem:err_est_Gal}
Let $u_h$ be the unique solution to \eqref{semi-discrete-equation}, 
and let $u_{hk}\in X_{hk}$ be the unique minimizer of~\eqref{minimisation_full_discrete}. Then, we have 
\begin{equation}
\|u_h-u_{hk}\|_{X_h} \le C \inf_{x_{hk} \in X_{hk}} \|u_h-x_{hk}\|_{X_h},
\end{equation}
where $C = \frac{\sqrt{3}M}{\beta}$ is independent of the time discretization.
\end{lemma}

\begin{proof}
Using  \eqref{Bh-isomorphism}, we have
\begin{align*}
\|u_{hk} - u_h\|_{X_h} &\le \beta^{-1} \Vert \mathcal{B}_h (u_{hk}-u_{h}) \Vert_{Z_h'} =\beta^{-1} \min_{x_{hk} \in X_{hk}} 
  \Vert \mathcal{B}_h (x_{hk} - u_h) \Vert_{Z_h'} \\&\le \sqrt{3}M\beta^{-1} \min_{x_{hk} \in X_{hk}} 
  \Vert   x_{hk} - u_{h} \Vert_{X_h},
\end{align*}
which proves the assertion.
\end{proof}

The unique minimizer of the quadratic discrete minimization problem~\eqref{minimisation_full_discrete} can be characterized by a system of linear equations. To write this system, let us first introduce 
the Riesz isomorphism $R_{V_h'}:V_{h}\rightarrow V_{h}'$ such that $\langle R_{V_h'}v_{h},w_{h} \rangle_{V_{h}',V_{h}}=\langle v_{h},w_{h} \rangle_{V}$ for all $v_{h},w_{h}\in V_{h}$ (note that 
$R_{V_h'}$ differs from the injection $J_{V_h}$ introduced above). Let $R_{Y_h'} : Y_h \rightarrow Y_h'$ be the space-time Riesz isomorphism such that
\begin{equation}
R_{Y_h'} = R_{V_h'} \otimes I_{L^2},
\end{equation}
where $I_{L^2}$ is the identity operator in $L^2(I)$ (it is actually the Riesz isomorphism from $L^2(I)$ onto $L^2(I)'\equiv L^2(I)$).
The quadratic discrete minimization problem \eqref{minimisation_full_discrete} is equivalent to the following linear problem: find $u_{hk}\in X_{hk}$ such that
\begin{equation}
\label{minimisation_equivalent}
B_{hk} u_{hk}=g_{hk},
\end{equation}
with $B_{hk}:X_{hk}\rightarrow X_{hk}'$ and $g_{hk}\in X_{hk}'$ such that, for all $x_{hk},z_{hk}\in X_{hk}$,
\begin{subequations}
\begin{align}
\label{bilinear_formulation_full_LHS}
\langle B_{hk} x_{hk}, z_{hk} \rangle_{X_{hk}',X_{hk}}&=
\langle \mathcal{A}_hx_{hk},R_{Y_h'}^{-1}\mathcal{A}_hz_{hk}\rangle_{Y_h',Y_h} + \alpha\langle x_{hk}(0),z_{hk}(0)\rangle_{L}, \\
\label{bilinear_formulation_full_RHS}
\langle g_{hk}, z_{hk} \rangle_{X_{hk}',X_{hk}}
&= \langle f_h,R_{Y_h'}^{-1}\mathcal{A}_hz_{hk}\rangle_{Y_h',Y_h} + \alpha\langle u_{0h},z_{hk}(0)\rangle_{L}.
\end{align}
\end{subequations}

Let us briefly describe the algebraic realization of the discrete problem~\eqref{minimisation_equivalent}. 
Let $(\psi_i)_{1\leq i\leq N_h}$ be a basis of $V_h$ and let $(\phi_l)_{1\leq l \leq N_k}$ be a basis of $S_k$. We can then seek for the components of the unique solution $u_{hk}$ of~\eqref{minimisation_equivalent} 
in the basis $(\psi_i\otimes \phi_l)_{1\leq i \leq N_h,1\leq l \leq N_k}$ of $X_{hk}$, i.e., we seek $\vecu=(\vecu_{il})_{1\leq i \leq N_h,1\leq l \leq N_k} \in \mathbb{R}^{N_hN_k}$ such that 
\begin{equation}
u_{hk}=\sum_{i=1}^{N_h}\sum_{l=1}^{N_k}\vecu_{il} \psi_i\otimes\phi_l.
\end{equation}
We define the following matrices of size $N_h\times N_h$ (related to the space discretization): 
\begin{equation} \label{eq:Dh_Mh}
(\bold{D}_h)_{ij}=\langle \psi_j , \psi_i \rangle_{V}, \quad\quad(\bold{M}_h)_{ij}=\langle \psi_j, \psi_i \rangle_{L},
\end{equation}
and the following matrices of size $N_k\times N_k$ (related to the time discretization):
\begin{equation} \label{eq:Dk_Mk_Ok}
(\bold{D}_k)_{lm}=\int_I \phi_m'(t)\phi_l'(t)dt, \quad(\bold{M}_k)_{lm}=\int_I \phi_m(t)  \phi_l(t)dt,
\quad (\bold{O}_k)_{lm}= \phi_m(0)  \phi_l(0).
\end{equation}
Recalling the separated form~\eqref{eq:separate}, we introduce the following matrix of size $N_h\times N_h$:
\begin{equation} \label{eq:Ahp}
(\bold{A}_h^{(p)})_{ij}=\langle A^{(p)}\psi_j,\psi_i \rangle_{V',V},
\end{equation}
and the following matrices of size $N_k\times N_k$:
\begin{equation} \label{eq:Mkpp_Ekp}
(\bold{M}_k^{(p,p')})_{lm}=\int_I \mu^{(p)}(t)\mu^{(p')}(t) \phi_m(t) \phi_l(t)dt,
\quad (\bold{E}_k^{(p)})_{lm}=\int_I  \mu^{(p)}(t) \phi_m'(t) \phi_l(t) dt,
\end{equation}
for all $1\le p,p'\le P$. Then, we obtain the following symmetric positive-definite linear system in $\mathbb{R}^{N_h N_k}$:
\begin{equation}\label{eq:discus}
\matB\vecu= \vecg,
\end{equation}
with the matrix
\begin{align} \label{eq:def_mathbb_B}
\matB = {}& \bold{M}_h\bold{D}_h^{-1}\bold{M}_h\otimes\bold{D}_k + \sum_{1\le p\le P} 2\mathrm{sym} \big\{ (\bold{A}_h^{(p)})^{\rm T}\bold{D}_h^{-1}\bold{M}_h\otimes \bold{E}_k^{(p)} \big\} \nonumber \\
&+ \sum_{1\le p,p'\le P} (\bold{A}_h^{(p)})^{\rm T}\bold{D}_h^{-1}\bold{A}_h^{(p')} \otimes \bold{M}_k^{(p,p')} + \alpha\bold{M}_h\otimes \bold{O}_k ,
\end{align}
where $\mathrm{sym}(\bold{Z}_h \otimes \bold{Z}_k) = \frac12(\bold{Z}_h \otimes \bold{Z}_k +\bold{Z}_h^{\mathrm{T}} \otimes \bold{Z}_k^{\mathrm{T}})$ for any matrix $\bold{Z}_h$ of size $N_h \times N_h$ and 
any matrix $\bold{Z}_k$ of size $N_k \times N_k$,
and the right-hand side
\begin{equation}
\vecg = \sum_{1\le q\le Q} \bold{M}_h \bold{D}_h^{-1} \bold{f}_h^{(q)}\otimes \bold{e}_k^{(q)} + \sum_{\substack{1\le p\le P \\1\le q\le Q}}
(\bold{A}_h^{(p)})^{\mathrm{T}} \bold{D}_h^{-1} \bold{f}_h^{(q)} \otimes \bold{d}_k^{(p,q)} + \alpha \bold{u_0}_{h} \otimes \bold{i}_k,
\end{equation}
with the vectors $(\bold{f}_h^{(q)})_i = \langle f^{(q)}|_{V_h}, \psi_i\rangle_{V',V}= \langle f^{(q)}, \psi_i\rangle_{V',V}$, $(\bold{u_0}_{h})_{i} = \langle u_{0h}, \psi_i \rangle_L = \langle u_0, \psi_i \rangle_L$,
for all $1\le i\le N_h$, 
and $(\bold{e}_k^{(q)})_l = \int_I \lambda^{(q)}(t) \phi_l'(t)dt$,
$(\bold{d}_k^{(p,q)})_l = \int_I \mu^{(p)}(t) \lambda^{(q)}(t) \phi_l(t)dt$,
$(\bold{i}_k)_{l} = \phi_l(0)$, 
for all $1\le l\le N_k$.

\begin{example}[Heat equation]
Let us consider the heat equation where $P=1$, $\mu^{(1)}(t)\equiv1$ and $A^{(1)}=-\Delta$, and let us equip the space $V$ with the $H^1$-seminorm so that $\langle v,w\rangle_{V} = \int_\Omega \nabla v(x){\cdot}\nabla w(x) dx = \langle A^{(1)}v,w\rangle_{V',V}$. Then the expression of $\matB$ simplifies as follows:
\begin{equation}
\matB = \bold{M}_h\bold{D}_h^{-1}\bold{M}_h\otimes\bold{D}_k 
+ \bold{M}_h\otimes 2\mathrm{sym}(\bold{E}_k)  + 
\bold{D}_h \otimes \bold{M}_k + \alpha\bold{M}_h\otimes \bold{O}_k ,
\end{equation}
with the following matrix of size $N_k\times N_k$:
\begin{equation}
(\bold{E}_k)_{lm}=\int_I \phi_m'(t) \phi_l(t) dt.
\end{equation}
\end{example}

\section{Low-rank approximation}
\label{sec:low_rank_approximation}

In this section, we present the low-rank approximation method we use to approximate iteratively the unique minimizer of~\eqref{minimisation_full_discrete} (or, equivalently, the unique solution to the linear system~\eqref{eq:discus}). We consider here a greedy algorithm~\cite{Temlyakov_2008_a,Cances_2011_a,Falco_2012_a,Chinesta_2014_a} which is an iterative procedure such that, at each iteration 
$m\in \mathbb{N}^*$, one computes an approximation $u_{hk}^m\in X_{hk}$ of the solution $u_{hk}\in X_{hk}$ of~\eqref{minimisation_full_discrete} in the form
\begin{equation}
\label{decomposition_u}
u_{hk}^m(x,t)=\sum_{1\le n\le m} v_{h}^n(x)\otimes s_{k}^n(t),
\end{equation}
with $v_{h}^n\in V_h$ and $s_{k}^n\in S_k$ for all $1\leq n \leq m$. 
The algorithm can be outlined as follows: 

\begin{framed}
GREEDY ALGORITHM:

\medskip

\begin{enumerate}
 \item Set $u_{hk}^0 = 0$ and $m=1$. 
 \item  \normalfont Solve for
 \begin{equation}
\label{minimisation_full}
(v_h^m,s_k^m)\in\underset{(v_{h},s_{k})\in V_h \times  S_k}{\textup{argmin}}\mathcal{E}_{h}(u_{hk}^{m-1} + v_h \otimes s_k).
\end{equation}
Set $u_{hk}^m:= u_{hk}^{m-1} + v_h^m \otimes s_k^m$. 
\item Check convergence, and if not satisfied, set $m\leftarrow m+1$ and go to step (2).
\end{enumerate}
\end{framed}
The following relative stagnation-based stopping criterion is used with a tolerance $\epsilon_{\rm greedy} >0$:
\begin{equation}
\frac{\|v_h^m\otimes s_k^m\|_{X}}{\|u_{hk}^m\|_{X}}<\epsilon_{\rm greedy}.
\end{equation}
Using the general results from \cite{Cances_2011_a,Falco_2012_a}, one can verify that the iterations of the above greedy algorithm are well-defined using the discrete minimal residual formulation presented in Section~\ref{sec:discrete}. 
Recall that the uniqueness of the solution to the minimization problem (\ref{minimisation_full_discrete}) follows from the strong convexity of the functional $\mathcal{E}_{h}$, and the sequence $(u_{hk}^m)_{m\in\mathbb{N}}$ converges to $u_{hk}$ as $n$ goes to infinity. Actually, it can be checked that this convergence result still holds true in the infinite-dimensional setting. 

In the above greedy algorithm, the minimization problem \eqref{minimisation_full} is nonlinear. Therefore, it is not straightforward to solve it and in practice, one often considers an alternating minimization algorithm (see \cite{Uschmajew:12}), based on the following fixed-point iterative scheme:

\medskip
\begin{framed}
ALTERNATING MINIMIZATION ALGORITHM FOR \eqref{minimisation_full}: 

\begin{enumerate}
 \item Choose $s_k^{m,0} \in S_k$ randomly and set $p=1$.
 \item Let $v_h^{m,p}\in V_h$ be the unique solution to 
\begin{equation}
\label{minimisation_v}
v_h^{m,p}=\underset{v_{h}\in V_h}{\textup{argmin}}\ \mathcal{E}_{h}\left(u_{hk}^{m-1} + v_{h}\otimes s_k^{m,p-1}\right).
\end{equation}
Compute $s_k^{m,p}\in S_k$ to be the unique solution to
\begin{equation}
\label{minimisation_w}
s_k^{m,p}=\underset{s_{k}\in S_k}{\textup{argmin}}\ \mathcal{E}_{h}\left(u_{hk}^{m-1} +  v_{h}^{m,p}\otimes s_k\right).
\end{equation}
\item Check convergence, and if not satisfied, set $p \leftarrow p+1$ and go to step (2).
\end{enumerate}
\end{framed}
The following relative stagnation-based stopping criterion is used with a tolerance $\epsilon_{\rm alt}>0$:
\begin{equation}
\frac{\|v_h^{m,p}\otimes s_k^{m,p}-v_h^{m,p-1}\otimes s_k^{m,p-1}\|_{X}}{\|v_h^{m,p}\otimes s_k^{m,p}\|_{X}}<\epsilon_{\rm alt}.
\end{equation}
The cost of an iteration of the alternating minimization algorithm is of order $(N_h+ N_k)$. Provided the number of fixed-point iterations remains reasonably small, the cost of each iteration of the greedy algorithm can be estimated to scale also as $(N_h+ N_k)$. We will verify in our numerical experiments that this is indeed the case.

\begin{remark}[Matrix form]
The matrix form of problem (\ref{minimisation_full}) is as follows: 
\begin{align*}
(\bold{v}_h^m,\bold{s}_k^m) =  \underset{(\bold{v}_{h},\bold{s}_{k})\in \mathbb{R}^{N_h} \times \mathbb{R}^{N_k}}{\textup{argmin}}&\bigg\{\frac{1}{2}\big(\vecu_{hk}^{m-1} + \bold{v}_{h} \otimes \bold{s}_{k}\big)^{\mathrm{T}}  \matB\big( \vecu_{hk}^{m-1} + \bold{v}_{h} \otimes \bold{s}_{k}\big) \\ &- \big( \vecu_{hk}^{m-1} + \bold{v}_{h} \otimes \bold{s}_{k}\big)^{\mathrm{T}} \vecg\bigg\},
\end{align*}
where $\vecu_{hk}^{m-1}$ denotes the vector in $\mathbb{R}^{N_h N_k}$ containing the coordinates of $u_{hk}^{m-1}$ in the basis $\left(\psi_i\otimes \phi_l\right)_{1\leq i \leq N_h, 1\leq l \leq N_k}$. Similarly, the matrix form of problems~\eqref{minimisation_v} and~\eqref{minimisation_w} is as follows:
\begin{align*}
\bold{v}_h^{m,p}=\underset{\bold{v}_{h}\in \mathbb{R}^{N_h}}{\textup{argmin}}&\bigg\{\frac{1}{2}\big(\vecu_{hk}^{m-1} + \bold{v}_{h} \otimes \bold{s}_{k}^{m,p-1}\big)^{\mathrm{T}}  \matB\big( \vecu_{hk}^{m-1} + \bold{v}_{h} \otimes \bold{s}_{k}^{m,p-1}\big) \\&- \big( \vecu_{hk}^{m-1} + \bold{v}_{h} \otimes \bold{s}_{k}^{m,p-1}\big)^{\mathrm{T}} \vecg\bigg\}, \\
\bold{s}_k^{m,p} = \underset{\bold{s}_{k}\in \mathbb{R}^{N_k}}{\textup{argmin}}&\bigg\{\frac{1}{2}\big(\vecu_{hk}^{m-1} + \bold{v}_{h}^{m,p} \otimes \bold{s}_{k}\big)^{\mathrm{T}}  \matB\big( \vecu_{hk}^{m-1} + \bold{v}_{h}^{m,p} \otimes \bold{s}_{k}\big) \\&- \big( \vecu_{hk}^{m-1} + \bold{v}_{h}^{m,p} \otimes \bold{s}_{k}\big)^{\mathrm{T}} \vecg\bigg\}.
\end{align*}
\end{remark}

\section{Other discrete minimal residual methods}
\label{sec:other}

In this section, we describe for the purpose of numerical comparisons in Section~\ref{sec:numerical_results} two other discrete minimal residual approaches. The discrete method introduced in Section~\ref{sec:discrete} is henceforth referred to as Method 1. The first variant, called Method 2, hinges on 
a fully discrete Petrov--Galerkin setting as devised in \cite{Andreev:12,Andreev_13,Andreev_2014_a}. The second variant, called Method 3, uses the same space semi-discrete setting as Method 1, but the test space is now equipped with a simple Euclidean norm of the components on a basis of $V_h$; Method 3 has been introduced in \cite{Nouy_2010_a}. 

\subsection{Method 2: fully discrete Petrov--Galerkin method}
\label{sec:PG}

Let us set
\begin{equation}
Y_{hk} := V_h\otimes S_k\uP, \qquad Y_{hk}'=V_h'\otimes (S_k\uP)',
\end{equation}
where $S_k\uP$ is a finite-dimensional subspace of $L^2(I)$ so that 
$S_k\uP\subset L^2(I)\equiv L^2(I)' \subset (S_k\uP)'$. 
The injection $J_{S_k\uP} : L^2(I) \rightarrow (S_k\uP)'$ is such that 
$J_{S_k\uP} = R_{(S_k\uP)'} \circ \Pi_{S_k\uP}$ where $\Pi_{S_k\uP}$ is the $L^2(I)$-orthogonal projection from $L^2(I)$ onto $S_k\uP$ and 
$R_{(S_k\uP)'} : S_k\uP \rightarrow (S_k\uP)'$ is the Riesz isomorphism so that 
$\langle R_{(S_k\uP)'}q,r\rangle_{(S_k\uP)',S_k\uP} = \langle q,r\rangle_{L^2(I)}$, for all $q,r\in S_k\uP$. 
Let us set $\textup{dim}(S_k\uP)=N_k\uP$. 
Recalling the separated form~\eqref{eq:separate}, let us define $f_{hk}\in Y_{hk}'$ s.t.
\begin{equation}
f_{hk}(t) = \sum_{1\le q\le Q} (J_{S_k\uP}\lambda^{(q)})(t) f^{(q)}|_{V_h}
\qquad \text{a.e.~$t\in I$}.
\end{equation} 
Let $\mathcal{A}_{hk}:= (I_{V_h'}\otimes J_{S_k\uP})\mathcal{A}_h : X_{hk} \rightarrow Y_{hk}'$ where $I_{V_h'}$ is the identity operator in $V_h'$ and $\mathcal{A}_h$ is defined by~\eqref{eq:def_Ah}. 
We consider the discrete energy functional $\mathcal{E}_{hk}\ufdPG:X_{hk}\rightarrow \mathbb{R}$ defined as
\begin{equation} \label{eq:def_calE_fdPG}
\mathcal{E}_{hk}\ufdPG(x_{hk}):=\frac12 
\left(\|\mathcal{A}_{hk}x_{hk}-f_{hk}\|_{Y_{hk}'}^2+\alpha\|x_{hk}(0)-u_{0h}\|_{L}^2\right), \qquad \forall x_{hk}\in X_{hk}.
\end{equation}
The discrete minimization problem is as follows: find $u_{hk}\ufdPG\in X_{hk}$ such that
\begin{equation}
\label{minimisation_full_discrete_PG}
u_{hk}\ufdPG= \underset{x_{hk}\in X_{hk}}{\textup{argmin}} \; \mathcal{E}_{hk}\ufdPG(x_{hk}).
\end{equation}

\begin{remark}[Comparison of energies]
Since $f_{hk}=(I_{V_h'}\otimes J_{S_k\uP})f_h$ with $f_h$ defined by~\eqref{eq:def_fh}, we have 
$$
\|\mathcal{A}_{hk}x_{hk}-f_{hk}\|_{Y_{hk}'} = \|(I_{V_h'}\otimes J_{S_k\uP})(\mathcal{A}_{h}x_{hk}-f_{h})\|_{Y_{hk}'} \leq \|\mathcal{A}_{h}x_{hk}-f_{h}\|_{Y_{h}'},
$$
which implies that $\mathcal{E}_{hk}\ufdPG(x_{hk}) \le \mathcal{E}_{h}(x_{hk})$ for all $x_{hk}\in X_{hk}$.
\end{remark}

As shown in \cite{Andreev_13,Andreev_2014_a} in the case of time-independent and selfadjoint operators $A\in\mathcal{L}(V;V')$, the Hessian of the discrete energy $\mathcal{E}_{hk}\ufdPG$ induces a bilinear form that satisfies an inf-sup condition that degenerates with the parabolic CFL. In the general case with a time-dependent differential operator, the positivity of the inf-sup constant is not guaranteed a priori, which means that the discrete energy functional $\mathcal{E}_{hk}\ufdPG$ may be only convex in some unfavorable situations. This means that in such cases, global minimizers of~\eqref{minimisation_full_discrete_PG} exist but may not be unique. Any minimizer satisfies
the following system of linear equations: find $u_{hk}\ufdPG\in X_{hk}$ such that
\begin{equation}
\label{minimisation_equivalent_PG}
B_{hk}\ufdPG u_{hk}\ufdPG=g_{hk}\ufdPG,
\end{equation}
with $B_{hk}\ufdPG:X_{hk}\rightarrow X_{hk}'$ and $g_{hk}\ufdPG\in X_{hk}'$ such that, for all $x_{hk},z_{hk}\in X_{hk}$,
\begin{subequations}
\begin{align}
\label{bilinear_formulation_full_LHS_PG}
\langle B_{hk}\ufdPG x_{hk}, z_{hk} \rangle_{X_{hk}',X_{hk}}&=
\langle \mathcal{A}_{hk}x_{hk},R_{Y_{hk}'}^{-1}\mathcal{A}_{hk}z_{hk}\rangle_{Y_{hk}',Y_{hk}} + \alpha\langle x_{hk}(0),z_{hk}(0)\rangle_{L}, \\
\label{bilinear_formulation_full_RHS_PG}
\langle g_{hk}\ufdPG, z_{hk} \rangle_{X_{hk}',X_{hk}}
&= \langle f_{hk},R_{Y_{hk}'}^{-1}\mathcal{A}_{hk}z_{hk}\rangle_{Y_{hk}',Y_{hk}} + \alpha\langle u_{0h},z_{hk}(0)\rangle_{L},
\end{align}
\end{subequations}
with the space-time Riesz isomorphism 
$R_{Y_{hk}'}= R_{V_h'} \otimes J_{S_k\uP}: Y_{hk} \rightarrow Y_{hk}'$. Let us point out that, since $J_{S_k\uP} = R_{(S_k\uP)'}$ on $S_k\uP$, we have $R_{Y_{hk}'}= R_{V_h'} \otimes R_{(S_k\uP)'}$. 
Let us briefly describe the algebraic realization of the discrete problem~\eqref{minimisation_equivalent_PG}. 
Recall that $(\psi_i)_{1\leq i\leq N_h}$ is a basis of $V_h$ and $(\phi_l)_{1\leq l \leq N_k}$ is a basis of $S_k$. Let $(\phi_l\uP)_{1\leq l \leq N_k\uP}$ be a basis of $S_k\uP$.
In addition to the square matrices $\bold{M}_h$, $\bold{D}_h$, $\bold{O}_k$ defined by~\eqref{eq:Dh_Mh}, \eqref{eq:Dk_Mk_Ok}, \eqref{eq:Ahp}, we consider the square matrix $\bold{M}_k\uP$ of size $N_k\uP\times N_k\uP$ and the rectangular matrix $\bold{E}_k\uPG$ of size $N_k\uP \times N_k$ such that
\begin{equation}
(\bold{M}_k\uP)_{lm} = \int_I \phi\uP_m(t)\phi\uP_l(t)dt,
\qquad
(\bold{E}_k\uPG)_{lm} = \int_I \phi'_m(t)\phi\uP_l(t)dt,
\end{equation}
and the rectangular matrices $\bold{M}_k\uPGp$, for all $1\le p\le P$, of size $N_k\uP \times N_k$ such that
\begin{equation}
(\bold{M}_k\uPGp)_{lm} = \int_I \mu^{(p)}(t)\phi_m(t)\phi\uP_l(t)dt.
\end{equation}
Then, we obtain the following symmetric positive-definite linear system in $\mathbb{R}^{N_h N_k}$:
\begin{equation}\label{eq:discus_PG}
\matB\ufdPG\vecu\ufdPG= \vecg\ufdPG,
\end{equation}
with the matrix
\begin{align}
\matB\ufdPG = {}& \bold{M}_h\bold{D}_h^{-1}\bold{M}_h\otimes (\bold{E}_k\uPG)^{\mathrm{T}} (\bold{M}_k\uP)^{-1} \bold{E}_k\uPG \nonumber \\
&+ \sum_{1\le p\le P} 2\mathrm{sym}\big\{ (\bold{A}_h^{(p)})^{\rm T}\bold{D}_h^{-1}\bold{M}_h\otimes (\bold{M}_k\uPGp)^{\mathrm{T}} (\bold{M}_k\uP)^{-1} \bold{E}_k\uPG \big\}  \nonumber \\ 
&+ \sum_{1\le p,p'\le P} (\bold{A}_h^{(p)})^{\rm T}\bold{D}_h^{-1}\bold{A}_h^{(p')} \otimes (\bold{M}_k\uPGp)^{\mathrm{T}} (\bold{M}_k\uP)^{-1} \bold{M}_k\uPGpp + \alpha\bold{M}_h\otimes \bold{O}_k,
\end{align}
and the right-hand side
\begin{align}
\vecg\ufdPG = {}& \sum_{1\le q\le Q} \bold{M}_h\bold{D}_h^{-1}\bold{f}_h^{(q)} \otimes (\bold{E}_k\uPG)^{\mathrm{T}} (\bold{M}_k\uP)^{-1} \bold{e}_k\uPq \nonumber\\&+ \sum_{\substack{1\le p\le P \\1\le q\le Q}}
(\bold{A}_h^{(p)})^{\mathrm{T}} \bold{D}_h^{-1} \bold{f}_h^{(q)} \otimes (\bold{M}_k\uPGp)^{\mathrm{T}} (\bold{M}_k\uP)^{-1} \bold{d}_k\uPq \nonumber \\&+ \alpha \bold{u_0}_h \otimes \bold{i}_k,
\end{align}
with $(\bold{e}_k\uPq)_l = \int_I \lambda^{(q)}(t) (\phi_l\uP)'(t)dt$ and
$(\bold{d}_k\uPq)_l = \int_I \lambda^{(q)}(t) \phi_l\uP(t)dt$,
for all $1\le l\le N_k\uP$.

\begin{remark}[Lowest-order Petrov--Galerkin discretization] \label{rem:CN}
Assume that $S_k$ is composed of continuous, piecewise affine functions and that $S_k\uP$ is composed of piecewise constant functions on the same time mesh so that $\textup{dim}(S_k\uP)=N_k-1$. This corresponds to the well-known Crank--Nicolson time scheme. Then, one can readily verify that
\begin{equation}
(\bold{E}_k\uPG)^{\mathrm{T}} (\bold{M}_k\uP)^{-1} \bold{E}_k\uPG = \bold{D}_k,
\qquad
(\bold{M}_k\uPGp)^{\mathrm{T}} (\bold{M}_k\uP)^{-1} \bold{E}_k\uPG = \bold{E}_k^{(p)},
\end{equation}
for all $1\le p\le P$, where $\bold{D}_k$ is defined in~\eqref{eq:Dk_Mk_Ok} and 
$\bold{E}_k^{(p)}$ in~\eqref{eq:Mkpp_Ekp}.
As a consequence, there is only one term composing the matrices $\matB$ and $\matB\ufdPG$ that differs, namely the time matrix in the double summation over $p,p'$ where this matrix is $\bold{M}_k^{(p,p')}$ for $\matB$ and is $(\bold{M}_k\uPGp)^{\mathrm{T}} (\bold{M}_k\uP)^{-1} \bold{M}_k\uPGpp$ for $\matB\ufdPG$. Even for the heat equation with $P=1$ and $\mu^{(1)}(t)\equiv1$, these matrices, which become, respectively, $\bold{M}_k$ and $(\bold{M}_k\uPG)^{\mathrm{T}} (\bold{M}_k\uP)^{-1} \bold{M}_k\uPG$ with
$(\bold{M}_k\uPG)_{lm} = \int_I \phi_m(t)\phi\uP_l(t)dt$, are still different. 
Note that $\bold{M}_k \ge (\bold{M}_k\uPG)^{\mathrm{T}} (\bold{M}_k\uP)^{-1} \bold{M}_k\uPG$ in the sense of quadratic forms, which is compatible with our above observation on the discrete energies that $\mathcal{E}_{hk}\ufdPG(x_{hk}) \le \mathcal{E}_{h}(x_{hk})$ for all $x_{hk}\in X_{hk}$.
As observed in \cite{Andreev:12,Andreev_13,Andreev_2014_a}, uniform stability with respect to the time discretization is not guaranteed, but this can
be fixed, e.g., by constructing the discrete test space $S_k\uP$ 
using a time mesh that is twice as fine as that used for the 
discrete trial space.
\end{remark}

\subsection{Method 3: an unpreconditioned space semi-discrete Galerkin method}

We consider the same space semi-discrete setting as in Section~\ref{sec:discrete} but we now equip the space $V_h$ with the Euclidean norm of the components on the basis $(\psi_i)_{1\le i\le N_h}$ instead of considering as before the norm induced by $V$. The main motivation for this change is that it avoids the appearance of the inverse stiffness matrix $\bold{D}_h^{-1}$ in the linear system.  We obtain the following symmetric positive-definite linear system in $\mathbb{R}^{N_hN_k}$:
\begin{equation}\label{eq:discus_3}
\matB\uun\vecu\uun= \vecg\uun,
\end{equation}
with the matrix
\begin{align}
\matB\uun = {}& (\bold{M}_h\bold{I}_h\bold{M}_h)\otimes\bold{D}_k + \sum_{1\le p\le P} 2\mathrm{sym} \big\{ ((\bold{A}_h^{(p)})^{\rm T}\bold{I}_h\bold{M}_h)\otimes \bold{E}_k^{(p)} \big\} \nonumber \\
&+ \sum_{1\le p,p'\le P} ((\bold{A}_h^{(p)})^{\rm T}\bold{I}_h\bold{A}_h^{(p')}) \otimes \bold{M}_k^{(p,p')} + \alpha\bold{M}_h\otimes \bold{O}_k ,
\end{align}
and the right-hand side
\begin{equation}
\vecg\uun = \sum_{1\le q\le Q} \bold{M}_h \bold{I}_h \bold{f}_h^{(q)}\otimes \bold{e}_k^{(q)} + \sum_{\substack{1\le p\le P \\1\le q\le Q}}
(\bold{A}_h^{(p)})^{\mathrm{T}} \bold{I}_h \bold{f}_h^{(q)} \otimes \bold{d}_k^{(p,q)} + \alpha \bold{u_0}_{h} \otimes \bold{i}_k,
\end{equation}
where $\bold{I}_h$ is the identity matrix of size $N_h\times N_h$. The present formulation is chosen for illustrative purposes; in practice, one can also replace $\bold{D}_h^{-1}$ by another matrix. 

\section{Numerical results}
\label{sec:numerical_results}

For all the test cases, we consider the space domain $\Omega=(0,1)\times(0,1)$, the time interval $I=[0,1]$, and the functional spaces $V=H^1_0(\Omega)$ and $L=L^2(\Omega)$. We consider first the heat equation, where the differential operator $A$ is time-independent and selfadjoint, then we consider a time-oscillatory diffusion problem, where the operator is time-dependent and selfadjoint, and finally a convection-diffusion equation, where the operator is time-independent and non-selfadjoint. The scaling factor for the contribution of the initial condition to the residual functional is always taken to be $\alpha:=1$. Let $\mathcal{T}_h$ be a shape-regular mesh of the domain $\Omega$; in what follows, we consider uniform meshes composed of square cells. The finite-dimensional finite element subspace $V_h\subset V$ of dimension $N_h$ is composed of continuous, piecewise bilinear functions on $\mathcal{T}_h$ vanishing at the boundary. Let $\mathcal{T}_k$ be a mesh of the interval $I$; for simplicity, we consider uniform meshes in time. The finite-dimensional subspace $S_k\subset H^1(I)$ of dimension $N_k$ is composed of continuous, piecewise affine functions on $\mathcal{T}_k$. 
In what follows, all the norms of residuals of algebraic quantities are evaluated using the Euclidean norm in $\mathbb{R}^{N_hN_k}$ which is denoted $\|\cdot\|_{\ell^2}$. When comparing to Method 2 (see Section~\ref{sec:PG}), we
considered the Crank--Nicolson time scheme discussed in Remark~\ref{rem:CN}. 
We also performed systematic comparisons with the uniformly stable variant 
using a finer time mesh for the test space, but we did not observe any 
significant difference in the results obtained for all the test cases 
considered herein.





\subsection{Test case 1: heat equation with manufactured solution}

We consider the heat equation with the time-independent, selfadjoint operator $A=-\Delta$. The initial condition is zero and the source term is evaluated from the following manufactured solution:
\begin{equation}
\label{manufactured_sol_1}
u(x,y,t)=\sum_{1\le n\le 10} n^{-4}\textup{sin}(\pi n^3 t)\textup{sin}(\pi n x)\textup{sin}(\pi n y).
\end{equation}
The discretization parameters are $N_h=(2^6)^2$ and $N_k=2^{13}$, and the stopping tolerances are $\epsilon_{\rm greedy}=10^{-5}$ and $\epsilon_{\rm alt}=5\times 10^{-2}$.

\begin{figure}[ht]
\begin{center}
\includegraphics[scale=0.2]{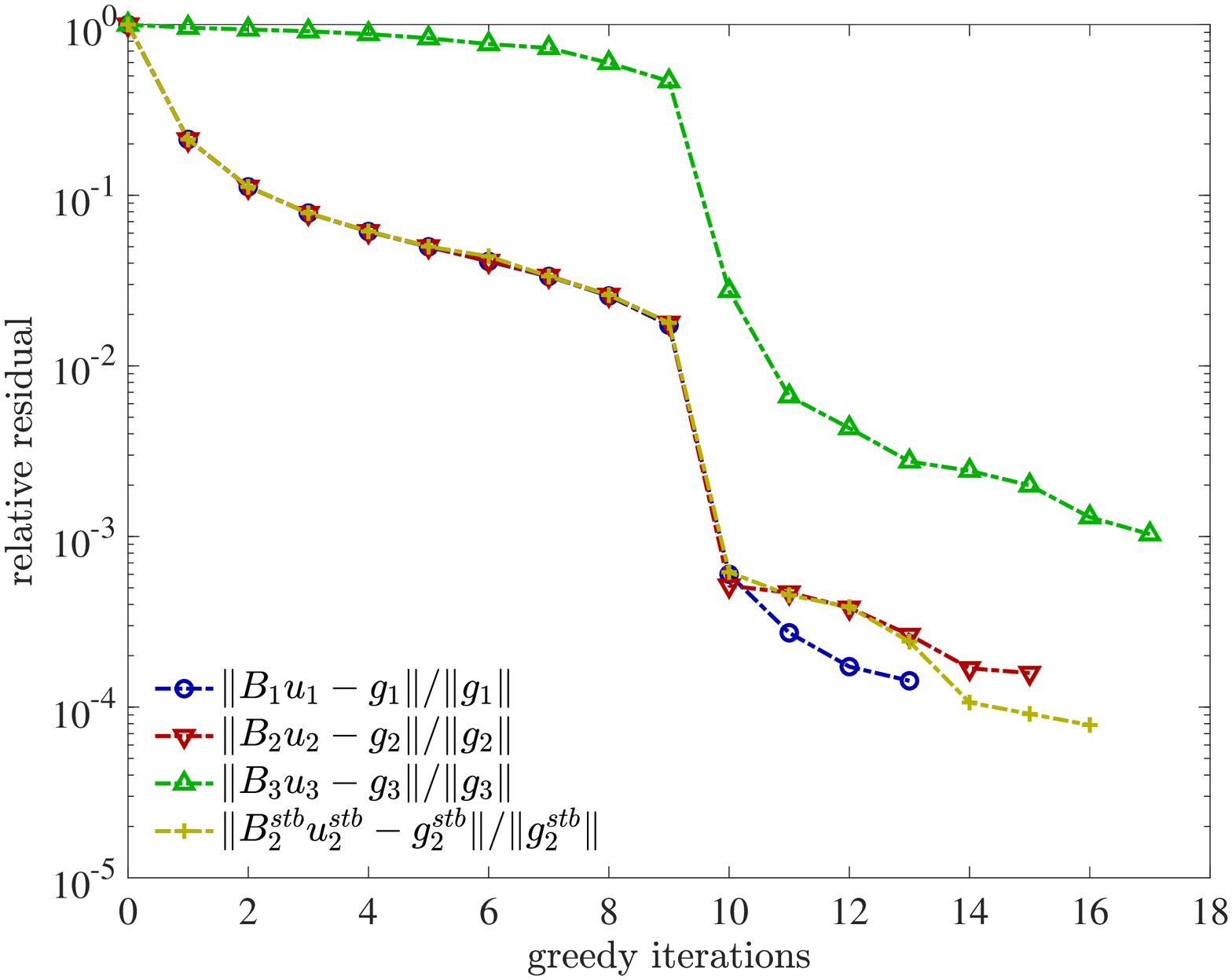} \qquad
\includegraphics[scale=0.2]{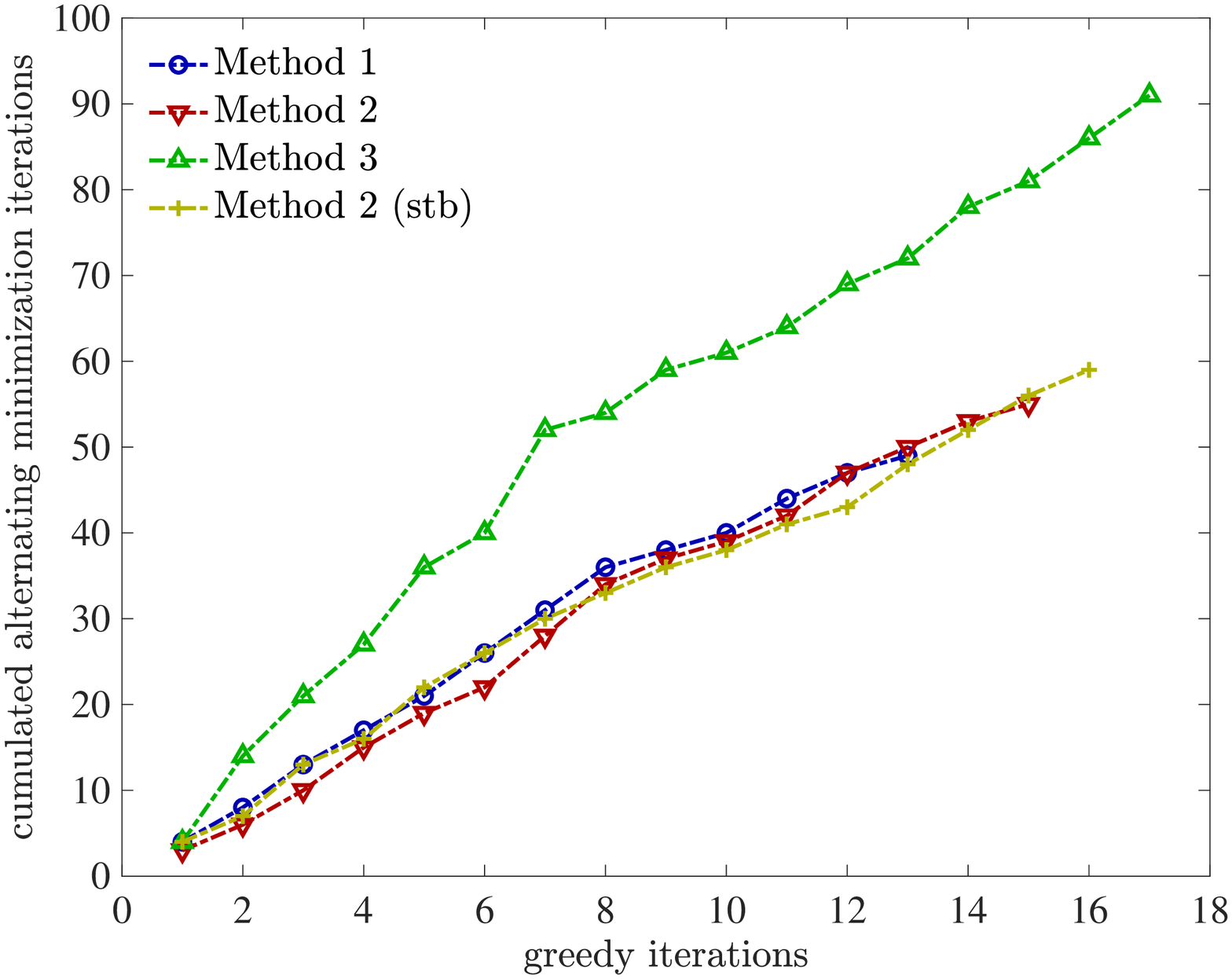}
\caption{Test case 1. Left: relative residual at each iteration of the greedy algorithm. Right: cumulated number of alternating minimization iterations in the greedy algorithm.}
\label{fig:heat_reserr}
\end{center}
\end{figure}

The left panel of Figure~\ref{fig:heat_reserr} presents the decrease of the relative residual as a function of the number of greedy iterations for Methods 1, 2 and 3. More precisely, we plot $\|\matB_i \vecu_i^m - \vecg_i\|_{\ell^2} /\|\vecg_i\|_{\ell^2}$ where $i\in\{1,2,3\}$ is the method index and $m$ is the greedy iteration counter. We notice that the three methods take about the same number of iterations (13, 15, and 17, respectively). This number is slightly larger than the space-time rank of the manufactured exact solution which is equal to 10. However, the relative residual takes larger values for Method 3 than for Methods 1 and 2. The right panel of Figure~\ref{fig:heat_reserr} presents the cumulated number of alternating minimization iterations in the greedy algorithm for Methods 1, 2 and 3. We observe that this number is about the same for Methods 1 and 2, whereas it is about 1.8 times larger for Method 3. Therefore, the use of a preconditioner, although it requires some additional computational effort, is beneficial to the efficiency of the overall behavior of the greedy algorithm. It is interesting to notice that with Methods 1 and 2, we have solved at convergence of the greedy algorithm about 50 linear systems in space, which is about $0.25\%$ of the amount that would have been solved by using an implicit time-stepping method (recall that $N_k=2^{13}$). We can make two further remarks concerning the decrease of the (relative) residual. First, we can decompose the residual of the space-time linear system as follows:
\begin{equation}
\matB_i \vecu_i^m - \vecg_i = (\matB_i^{\mathrm{pde}} \vecu_i^m - \vecg_i^{\mathrm{pde}}) + (\matB_i^{\mathrm{ic}} \vecu_i^m - \vecg_i^{\mathrm{ic}}),
\end{equation}
where we have written $\matB_i = \matB_i^{\mathrm{pde}} + \matB_i^{\mathrm{ic}}$ and $\vecg_i = \vecg_i^{\mathrm{pde}} + \vecg_i^{\mathrm{ic}}$ to distinguish the contribution of the differential operator from that of the initial condition. Our results (not displayed for brevity) show that after a few greedy iterations, the two contributions have about the same size. Moreover, as the greedy iteration approaches convergence, there is some compensation between the two contributions to the relative residual in Method 1 (but not for Method 2) since they have a size which is about one order of magnitude larger than the relative residual itself. As a further comparison, we considered the quantities 
\begin{equation} \label{eq:def_rim}
r_i^m = \|\matB_1 \vecu_i^m - \vecg_1\|_{\ell^2} / \|\vecg_1\|_{\ell^2},
\end{equation} 
which represent the relative residual for the linear system originating from Method 1 when the iterates produced by Method $i\in\{1,2,3\}$ are inserted into the residual. As expected from the MinRes formulation, $r_1^m\le \min(r_2^m,r_3^m)$ for all $m\ge0$, and as the greedy iterations approach convergence, $r_1^m$ reaches the value $4\times 10^{-5}$, whereas $r_2^m$ and $r_3^m$ reach a value of $4\times10^{-4}$ and $10^{-4}$, respectively.  

\begin{figure}[ht]
\begin{center}
\includegraphics[scale=0.2]{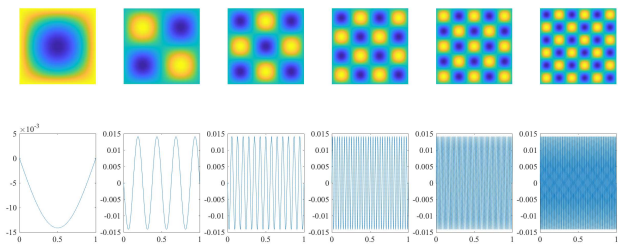} 
\caption{Test case 1: first six modes in space (top row) and in time (bottom row) for Method 1.}
\label{fig:heat_modes}
\end{center}
\end{figure}

Figure~\ref{fig:heat_modes} presents the first six space and time modes for Method 1. The first six modes obtained with Method 2 are essentially the same, whereas some differences, especially in the space modes, can be observed with Method 3. This indicates that the preconditioner plays a relevant role in the exploration of the discrete trial space.  

\begin{figure}[ht]
\begin{center}
\includegraphics[scale=0.2]{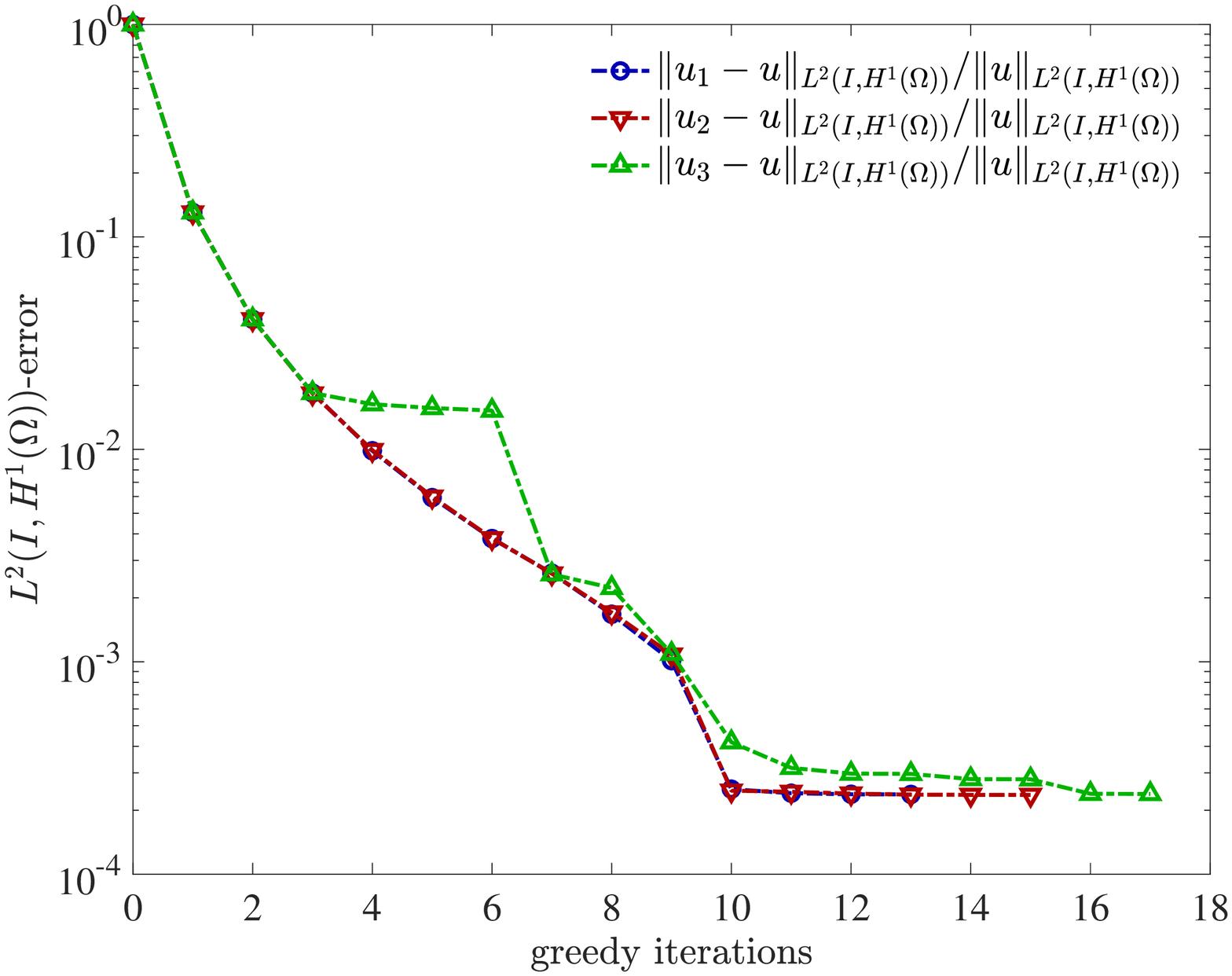} \qquad
\includegraphics[scale=0.2]{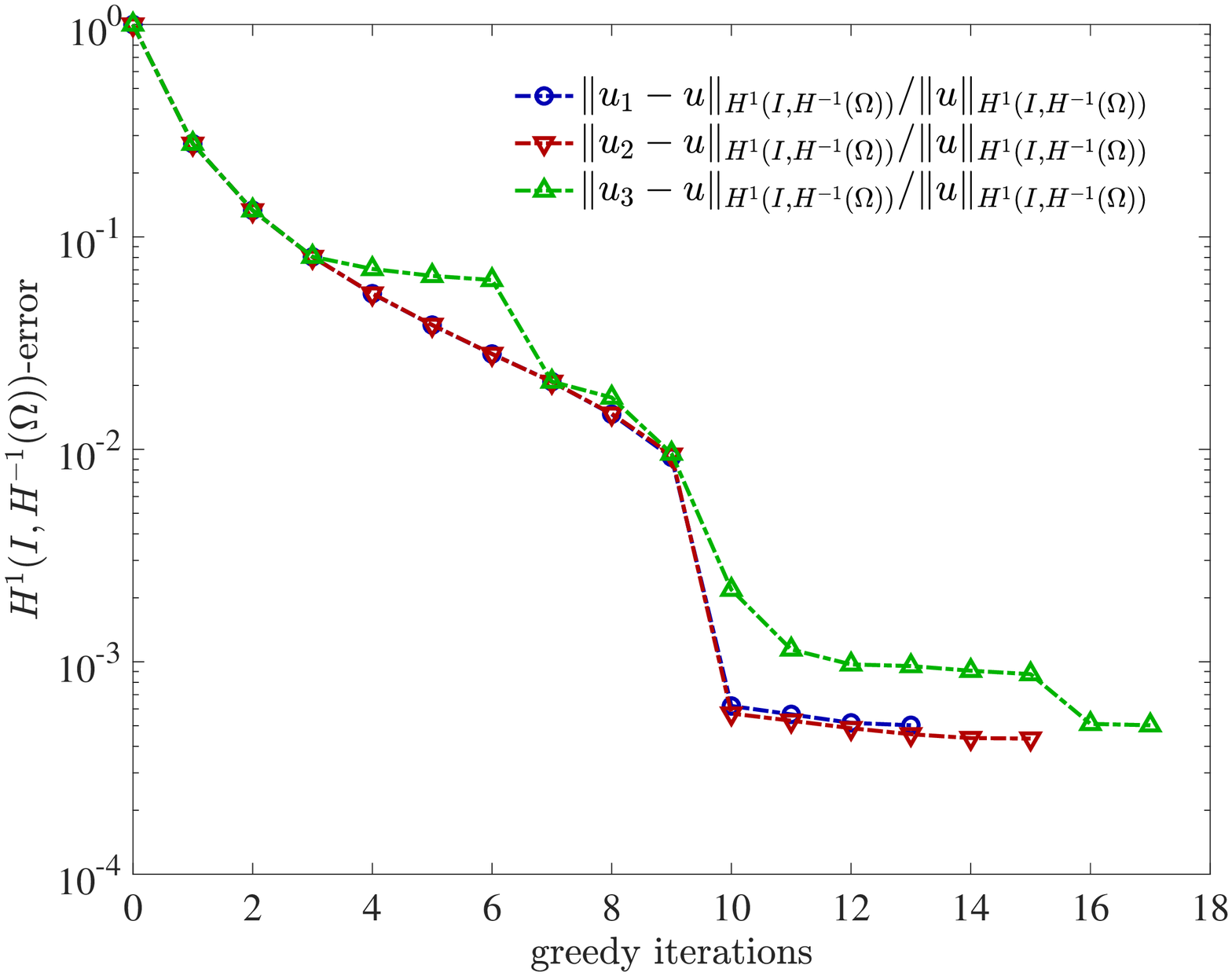} 
\caption{Test case 1: comparison of the errors produced by Methods 1, 2, 3 in two norms: $L^2(I;H^1(\Omega))$ (left) and $H^1(I;H^{-1}(\Omega))$ (right); in both cases, the curves for Methods~1 and~2 almost overlap.}
\label{fig:heat_err}
\end{center}
\end{figure}

Figure~\ref{fig:heat_err} reports the normalized errors $(u_{hk,i}^m - u)$ as a function of the iteration counter $m$ of the greedy algorithm where, as above, the additional subscript $i\in\{1,2,3\}$ indicates which method has been used. The errors are measured in the $L^2(I;H^1(\Omega))$- and $H^1(I;H^{-1}(\Omega))$-norms. The three methods produce in both norms relatively close errors, and the error for Method 3 is always a bit larger. At convergence of the greedy algorithm, both errors are very small, namely $2\times 10^{-4}$ in the $L^2(I;H^1(\Omega))$-norm and $5\times 10^{-4}$ in the $H^1(I;H^{-1}(\Omega))$-norm.

\begin{figure}[ht]
\begin{center}
\includegraphics[scale=0.2]{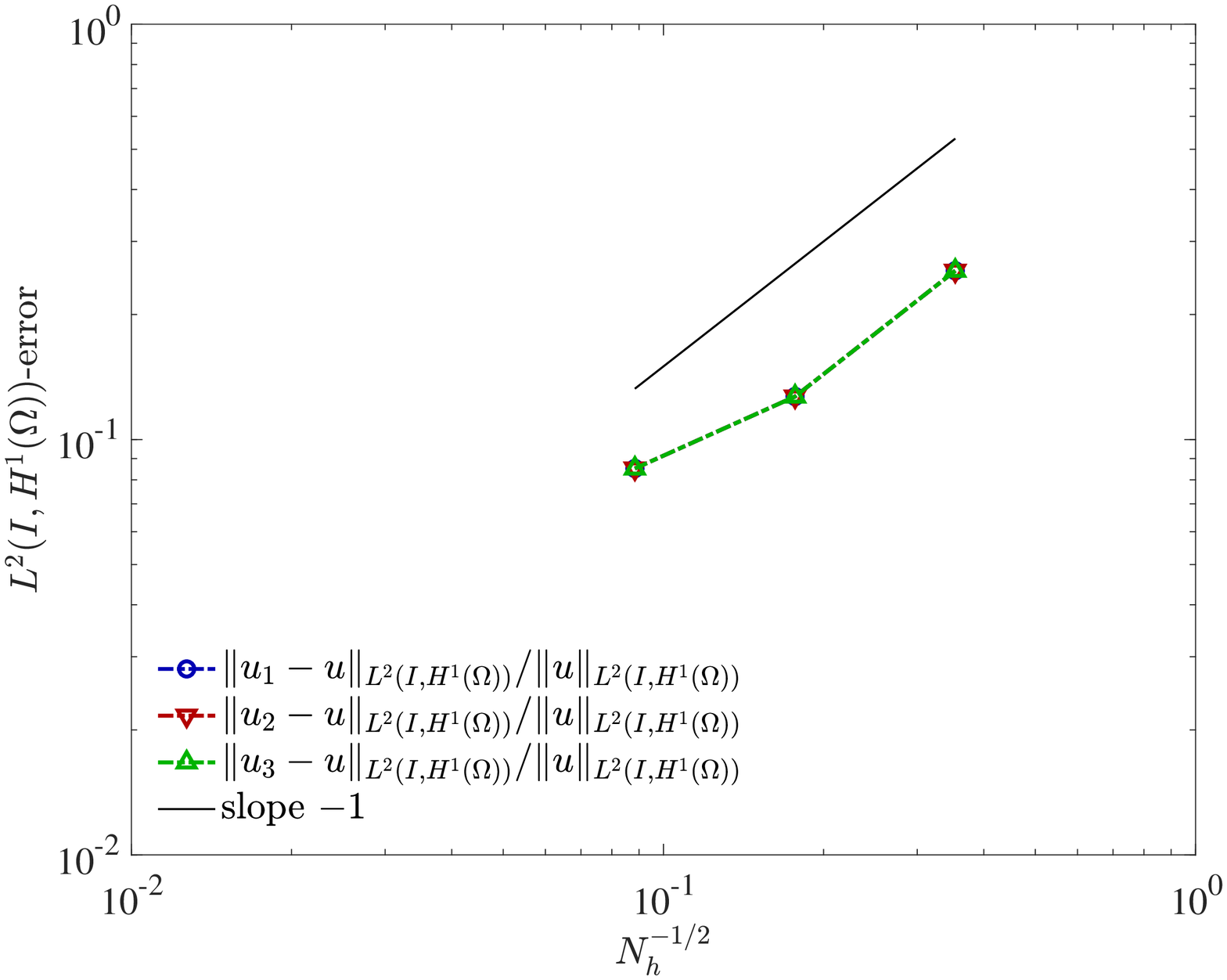} \qquad
\includegraphics[scale=0.2]{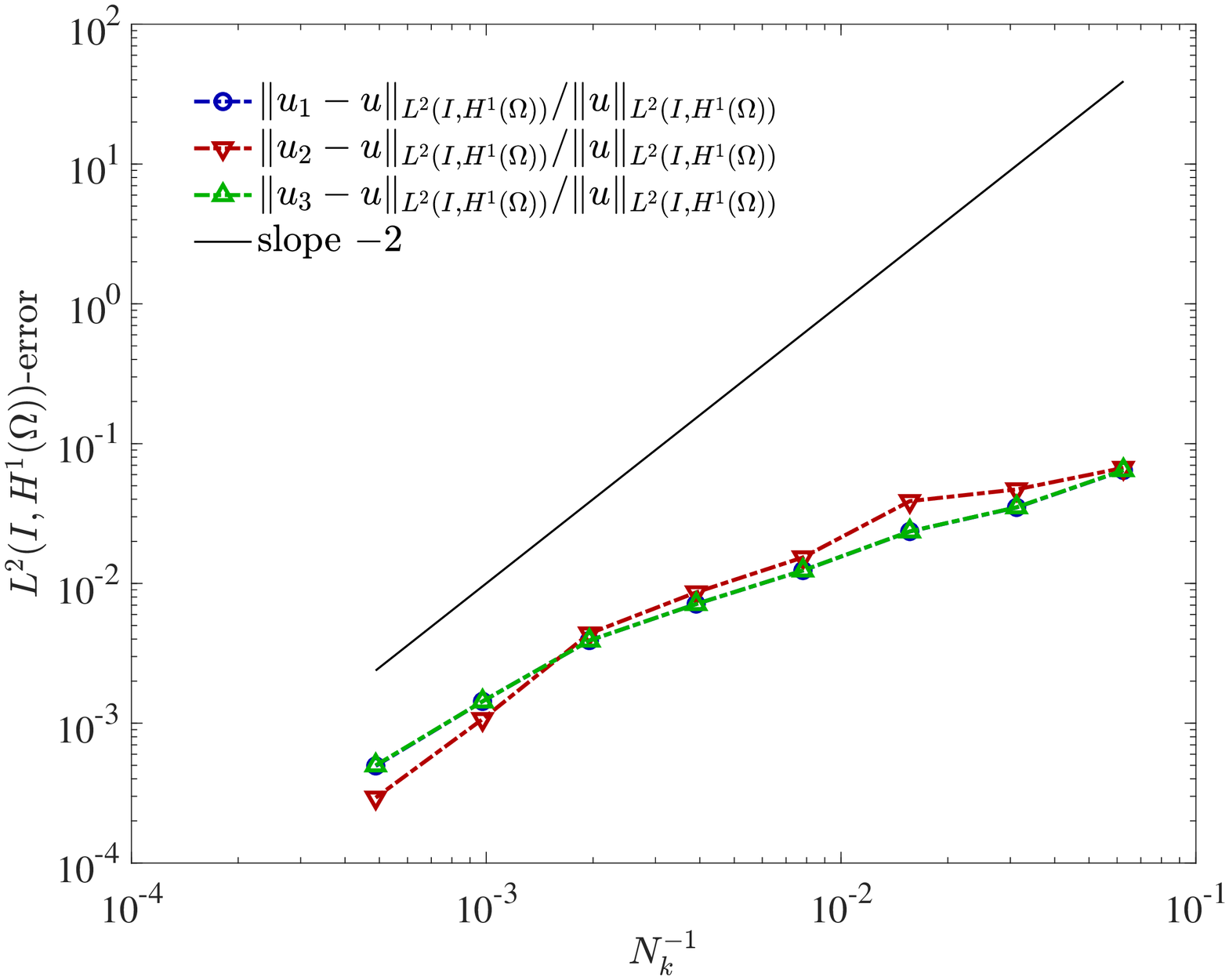}  \\
\includegraphics[scale=0.2]{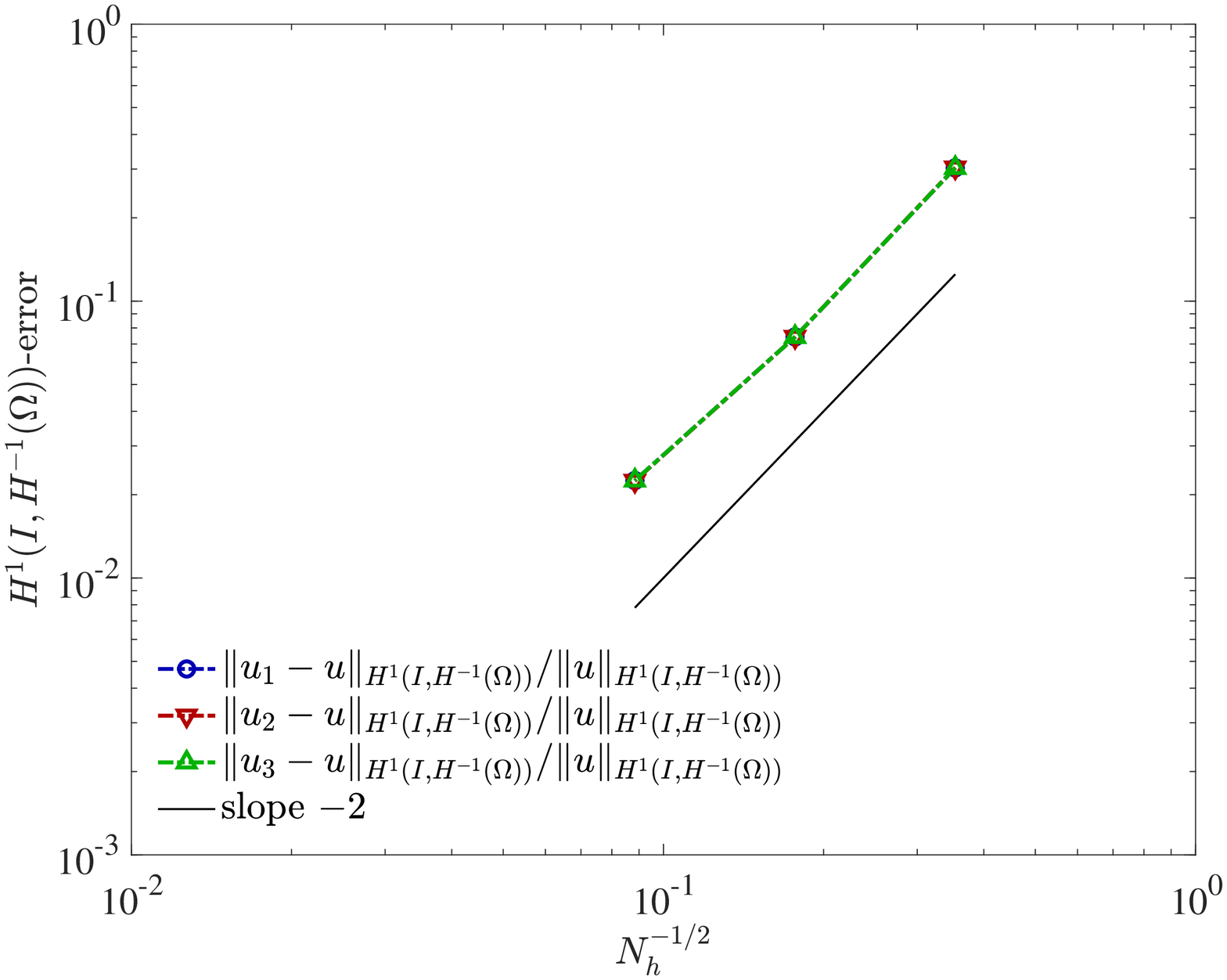} \qquad
\includegraphics[scale=0.2]{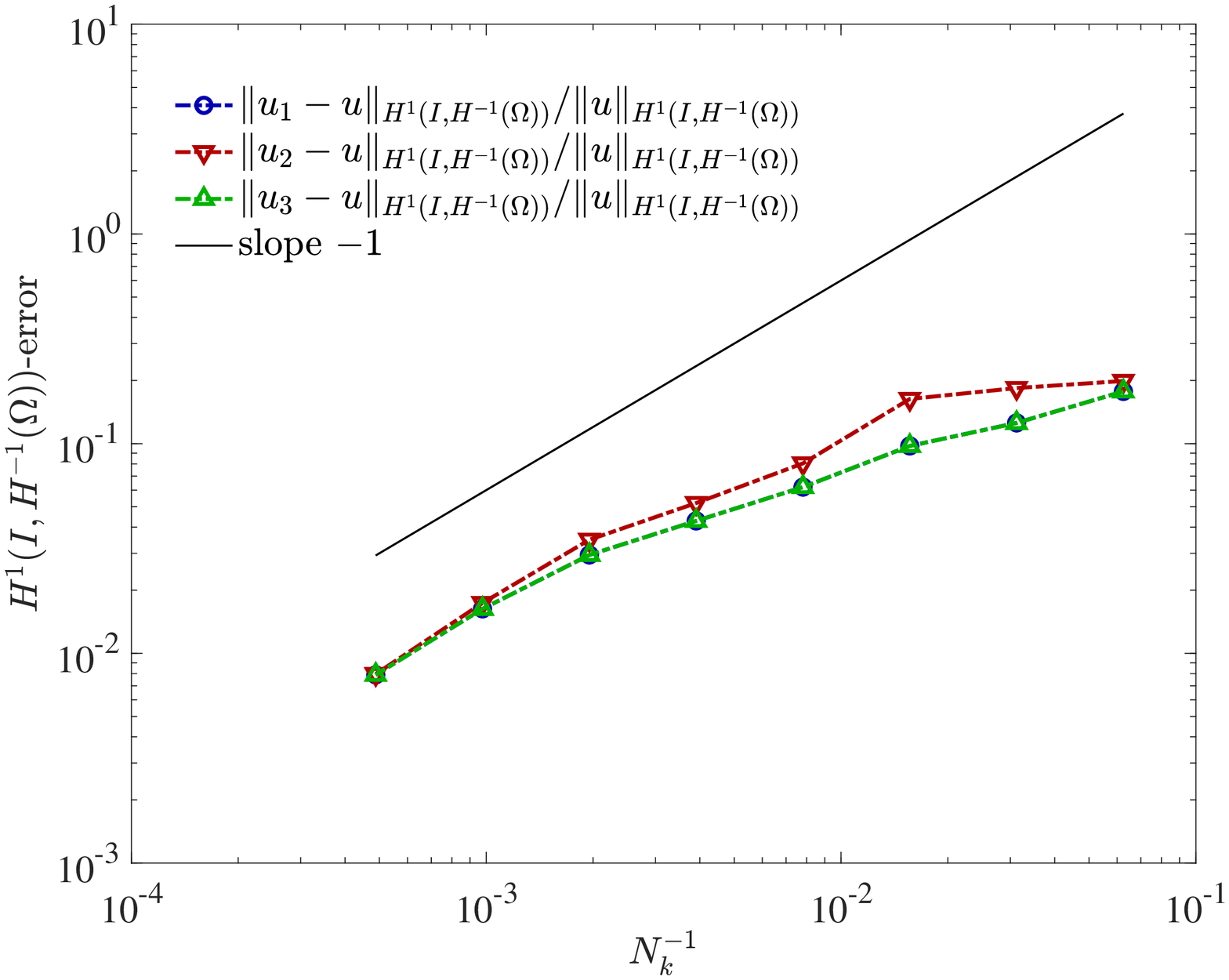}
\caption{Test case 1: convergence study for Methods 1, 2 and 3 for errors measured in the $L^2(I;H^1(\Omega))$-norm (top row) and in the $H^1(I;H^{-1}(\Omega))$-norm (bottom row) for various mesh-sizes $N_h^{-1/2}$ (left column) and various time-steps $N_k^{-1}$ (right column); in all cases, the curves for Methods~1 and~3 overlap.}
\label{fig:heat_conv}
\end{center}
\end{figure}

Figure~\ref{fig:heat_conv} presents a convergence study for Methods 1, 2 and 3 as a function of the discretization parameters $N_h$ (in space) and $N_k$ (in time). In both cases, we report the $L^2(I;H^1(\Omega))$- and $H^1(I;H^{-1}(\Omega))$-errors. The left panel considers $N_h=(2^l)^2$, $l\in\{2,3,4\}$ with $N_k=2^{13}$, whereas the right panel considers $N_k=2^l$, $l\in\{4,\ldots,11\}$ with $N_h=(2^6)^2$. We observe that the convergence is of second order in time (if the time-step is small enough) and first order in space in the $L^2(I;H^1(\Omega))$-norm, whereas it is of first order in time (if the time-step is small enough) and second order in space in the $H^1(I;H^{-1}(\Omega))$-norm. These convergence orders are consistent with the expected decay rates of the best-approximation errors in both norms when approximating smooth functions by elements of the discrete trial space $X_{hk}$. Incidentally, we observe that the errors produced by Method 2 in both norms are slightly worse for the coarser time discretizations; this observation is consistent with the CFL-dependent inf-sup stability estimate for Method 2.


\subsection{Test case 2: time-dependent diffusion}

We consider a time-dependent, selfadjoint differential operator $A(t)=-\mu(t)\Delta$ with diffusion coefficient $\mu(t)=\sin(100\pi t)+2$. The initial condition is $u_0=0$ and the source term is $f=1$. The explicit expression of the exact solution is not available. The discretization parameters are $N_h=(2^6)^2$ and $N_k=2^{13}$ (as in the previous test case), and the stopping tolerances are $\epsilon_{\rm greedy}=10^{-5}$ and $\epsilon_{\rm alt}=5\times 10^{-2}$ (as in the previous test case).

\begin{figure}[ht]
\begin{center}
\includegraphics[scale=0.2]{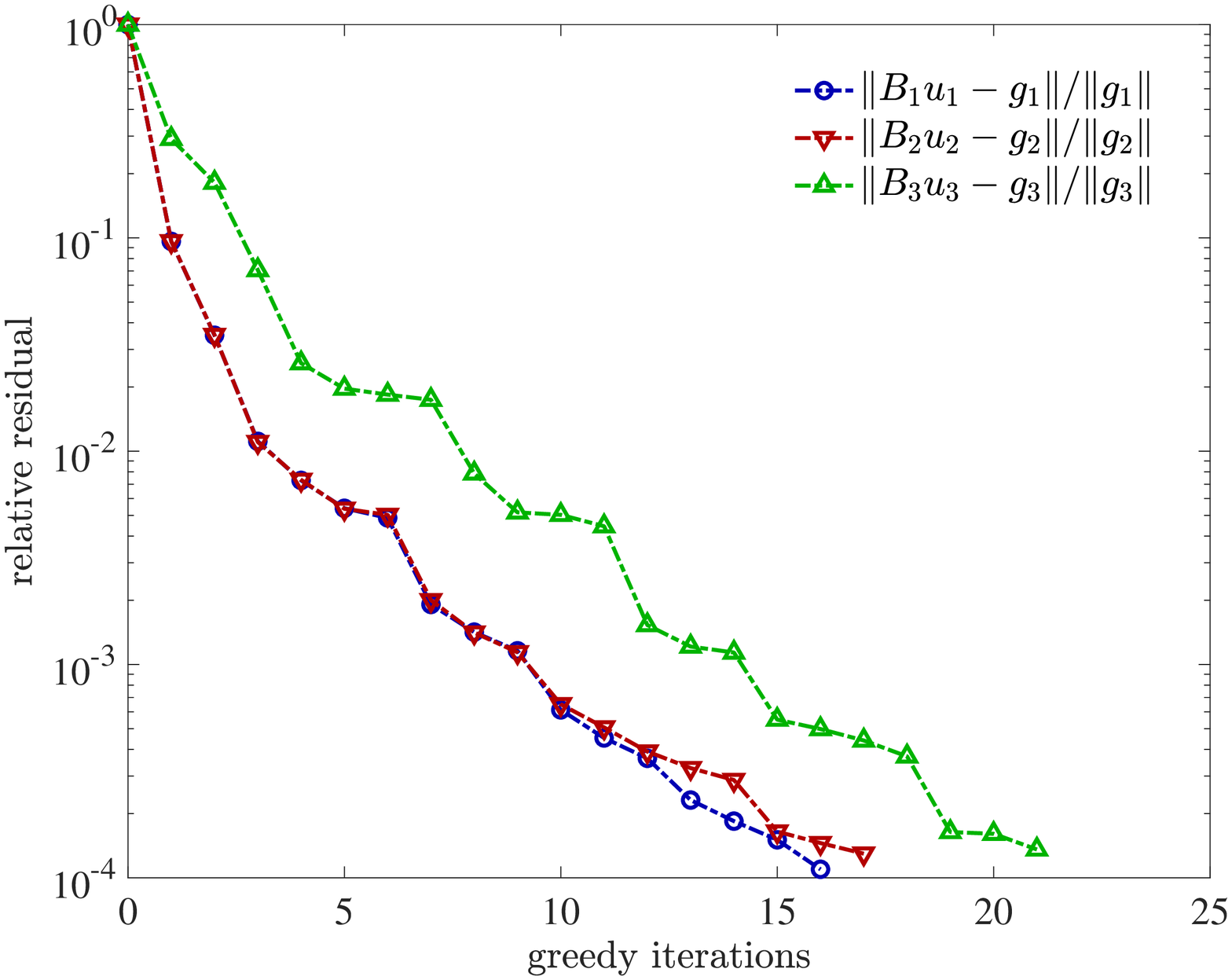} \qquad
\includegraphics[scale=0.2]{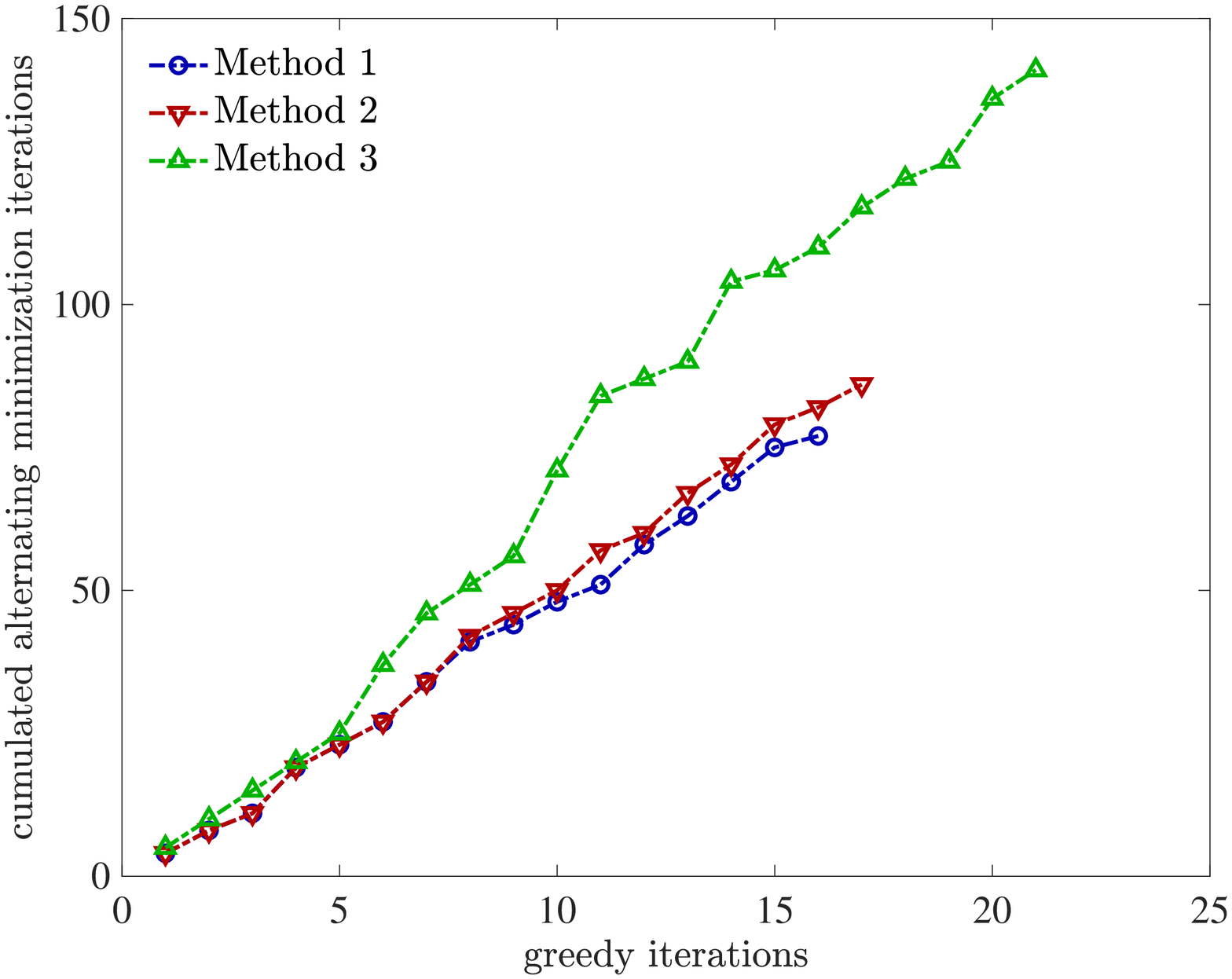}
\caption{Test case 2. Left: relative residual at each iteration of the greedy algorithm. Right: cumulated number of alternating minimization iterations in the greedy algorithm.}
\label{fig:time_reserr}
\end{center}
\end{figure}

The left panel of Figure~\ref{fig:time_reserr} presents the decrease of the relative residual as a function of the number of greedy iterations for Methods 1, 2 and 3. We notice that the greedy algorithm takes between 16 and 21 iterations for the three methods to converge. The right panel of Figure~\ref{fig:time_reserr} presents the cumulated number of alternating minimization iterations in the greedy algorithm for Methods 1, 2 and 3. We observe that this number is about the same for Methods 1 and 2 (as for test case 1), whereas it is about 1.5 times larger for Method 3, confirming once again the benefit of using a preconditioner. It is interesting to notice that with Methods 1 and 2, we have solved at convergence of the greedy algorithm about 80 linear systems in space, which is about $0.5\%$ of the amount that would have been solved by using an implicit time-stepping method (recall that $N_k=2^{13}$).
Furthermore, similar observations as for test cases 1 and 2 can be made concerning the two contributions to the relative residual and the behavior of the residuals $r_i^m$ defined by~\eqref{eq:def_rim}. In particular, we have again 
$r_1^m\le \min(r_2^m,r_3^m)$ for all $m\ge0$ (as expected from the MinRes formulation); as the greedy algorithm approaches convergence, $r_1^m$ reaches a value of $9\times 10^{-5}$, whereas $r_2^m$ and $r_3^m$ reach a value of $10^{-4}$ and $2\times 10^{-4}$, respectively.

\begin{figure}[ht]
\begin{center}
\includegraphics[scale=0.2]{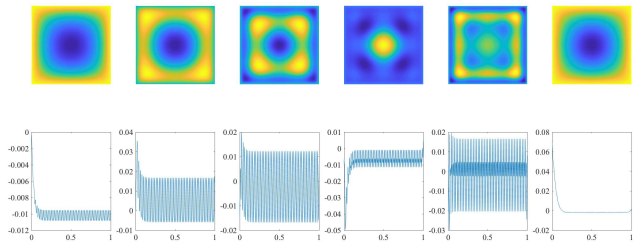} 
\caption{Test case 2: first six modes in space (top row) and in time (bottom row) for Method 1.}
\label{fig:time_modes}
\end{center}
\end{figure}

Figure~\ref{fig:time_modes} presents the first six space and time modes for Method 1. The first six modes obtained with Method 2 are essentially the same, whereas some differences, especially in the space modes, can be observed with Method 3. This observation again confirms that the preconditioner plays a relevant role in the exploration of the discrete trial space.  

\begin{figure}[ht]
\begin{center}
\includegraphics[scale=0.2]{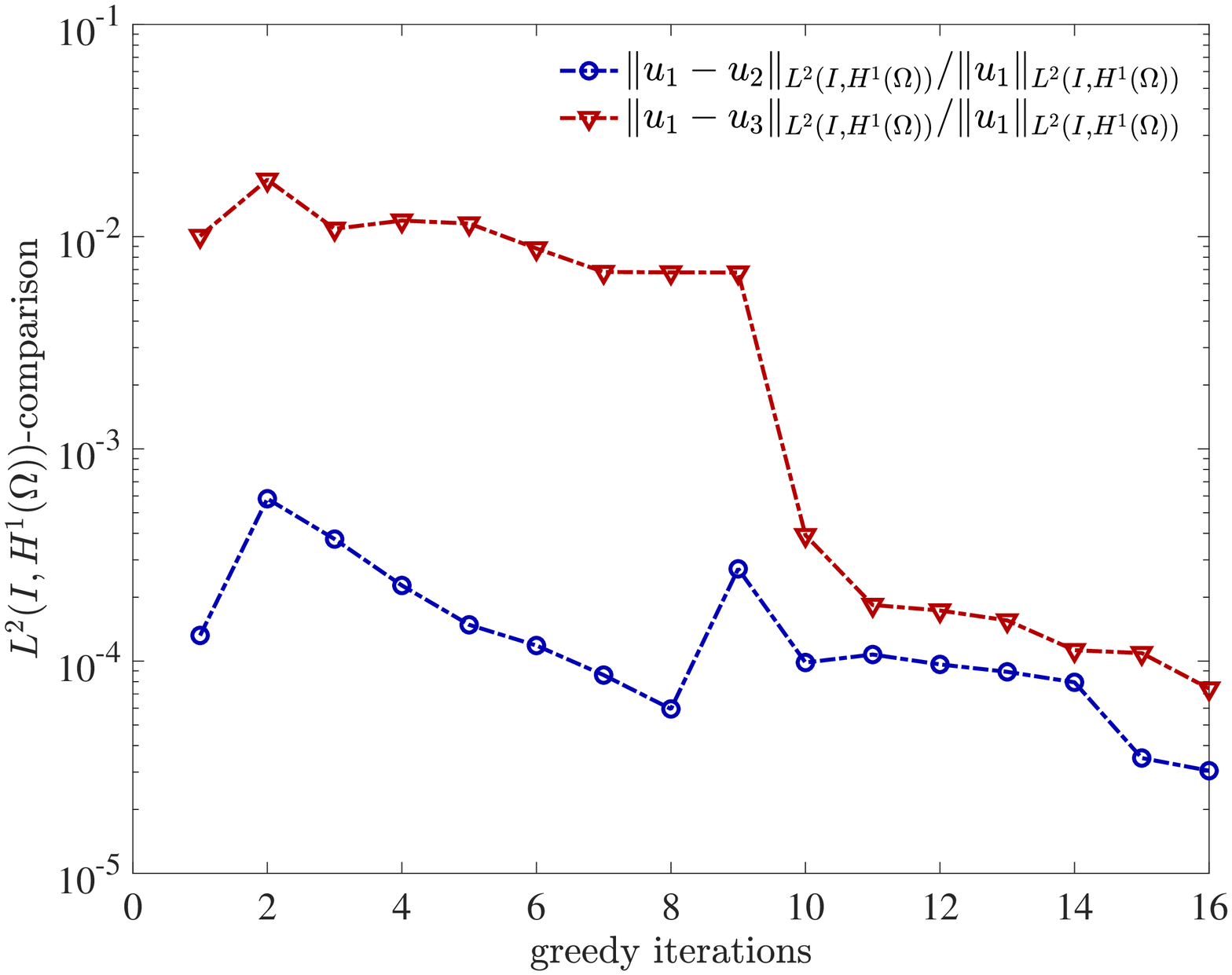} \qquad
\includegraphics[scale=0.2]{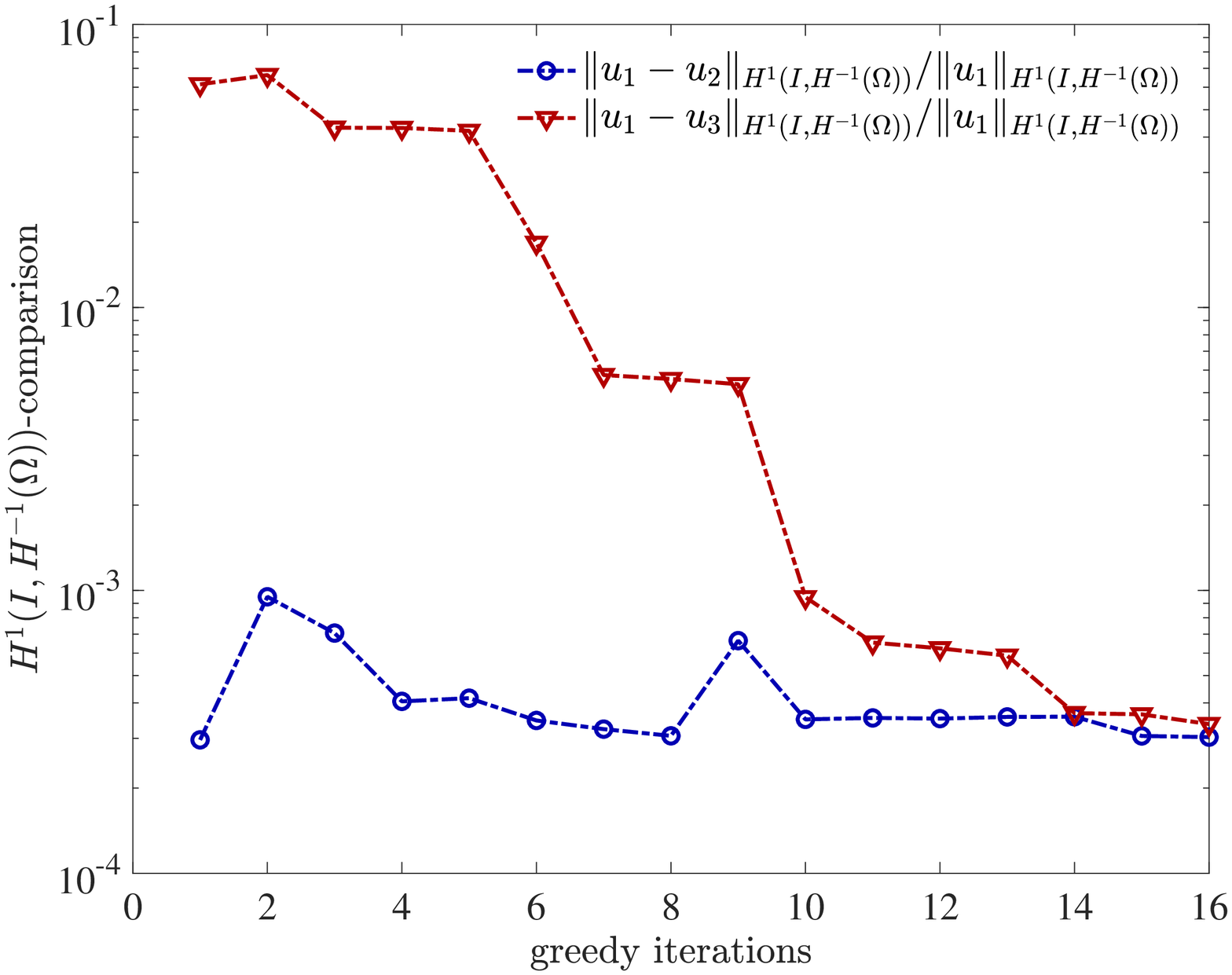} 
\caption{Test case 2: comparison of the solutions produced by Methods 1, 2, 3 in two norms: $L^2(I;H^1(\Omega))$ (left) and $H^1(I;H^{-1}(\Omega))$ (right).}
\label{fig:time_err}
\end{center}
\end{figure}

Figure~\ref{fig:time_err} reports the normalized differences $(u_{hk,1}^m-u_{hk,2}^m)$ and $(u_{hk,1}^m-u_{hk,3}^m)$ as a function of the iteration counter $m$ of the greedy algorithm, where, as above, the additional subscript $i\in\{1,2,3\}$ indicates which method has been used. These differences are measured in the $L^2(I;H^1(\Omega))$- and $H^1(I;H^{-1}(\Omega))$-norms. We observe that the three methods produce approximate solutions that are relatively close in both norms. At the convergence of the greedy algorithm, the difference in the $L^2(I;H^1(\Omega))$-norm is of the order of $5\times 10^{-5}$, and it is about one order of magnitude higher in the $H^1(I;H^{-1}(\Omega))$-norm.

\begin{figure}[ht]
\begin{center}
\includegraphics[scale=0.2]{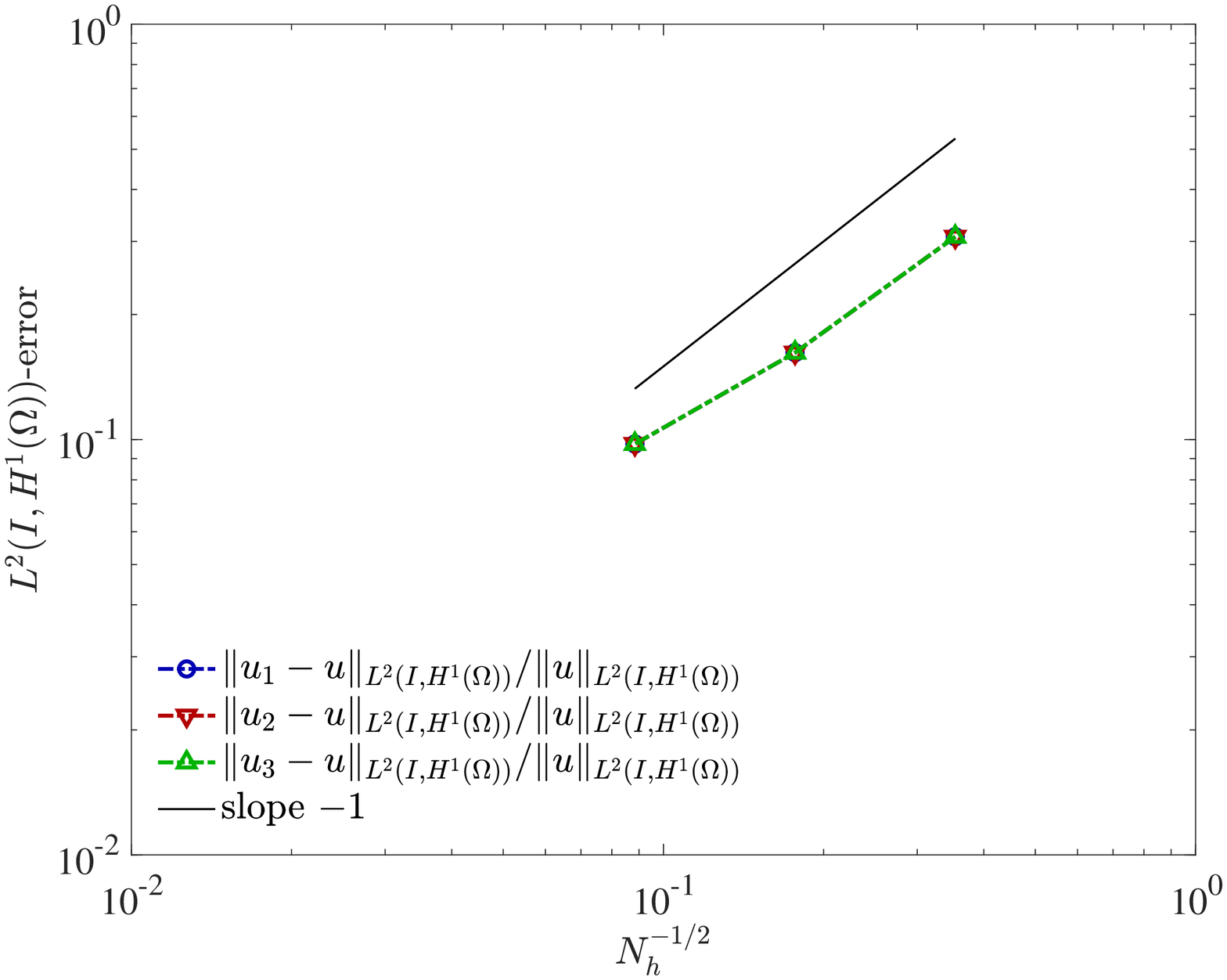} \qquad
\includegraphics[scale=0.2]{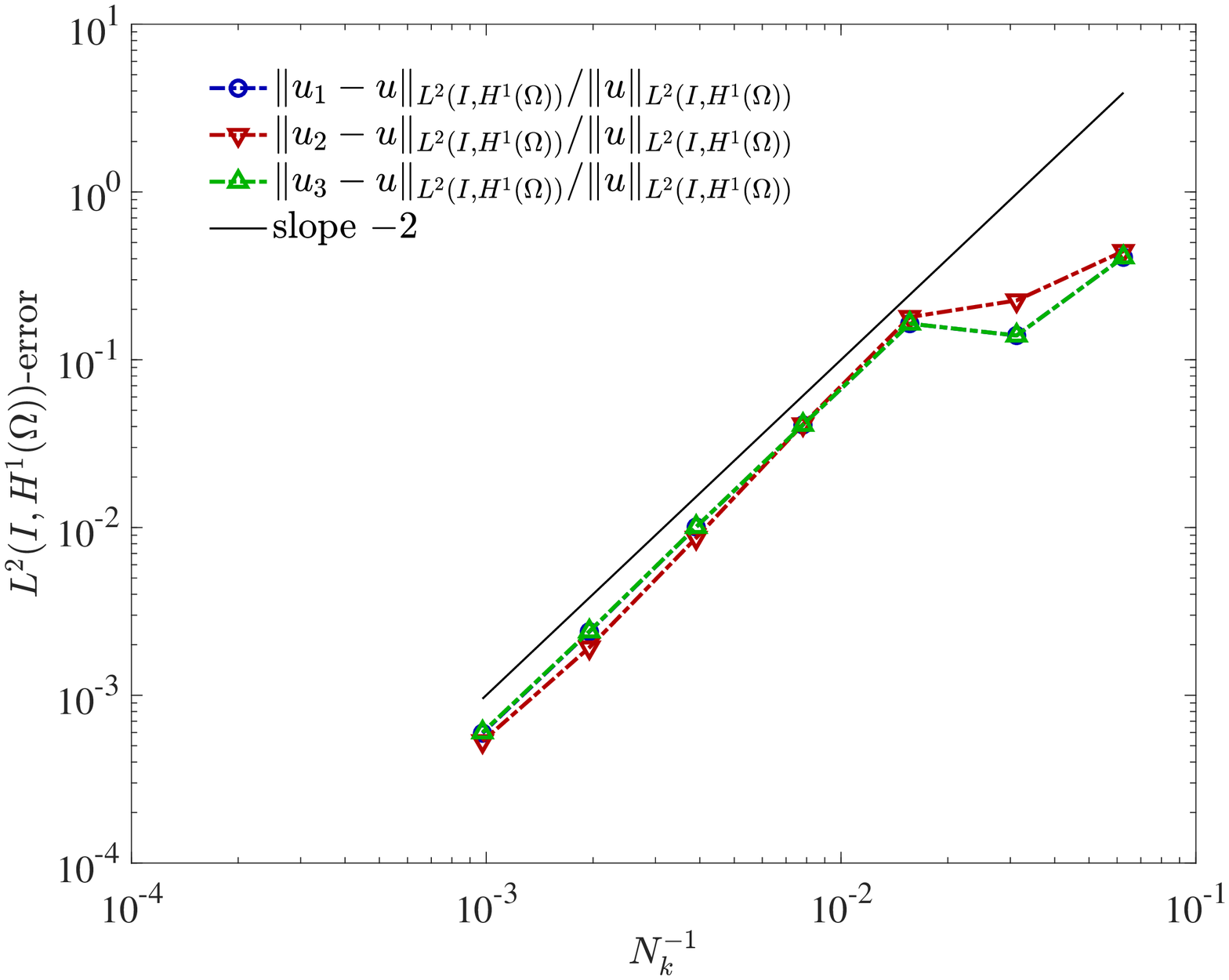} \\
\includegraphics[scale=0.2]{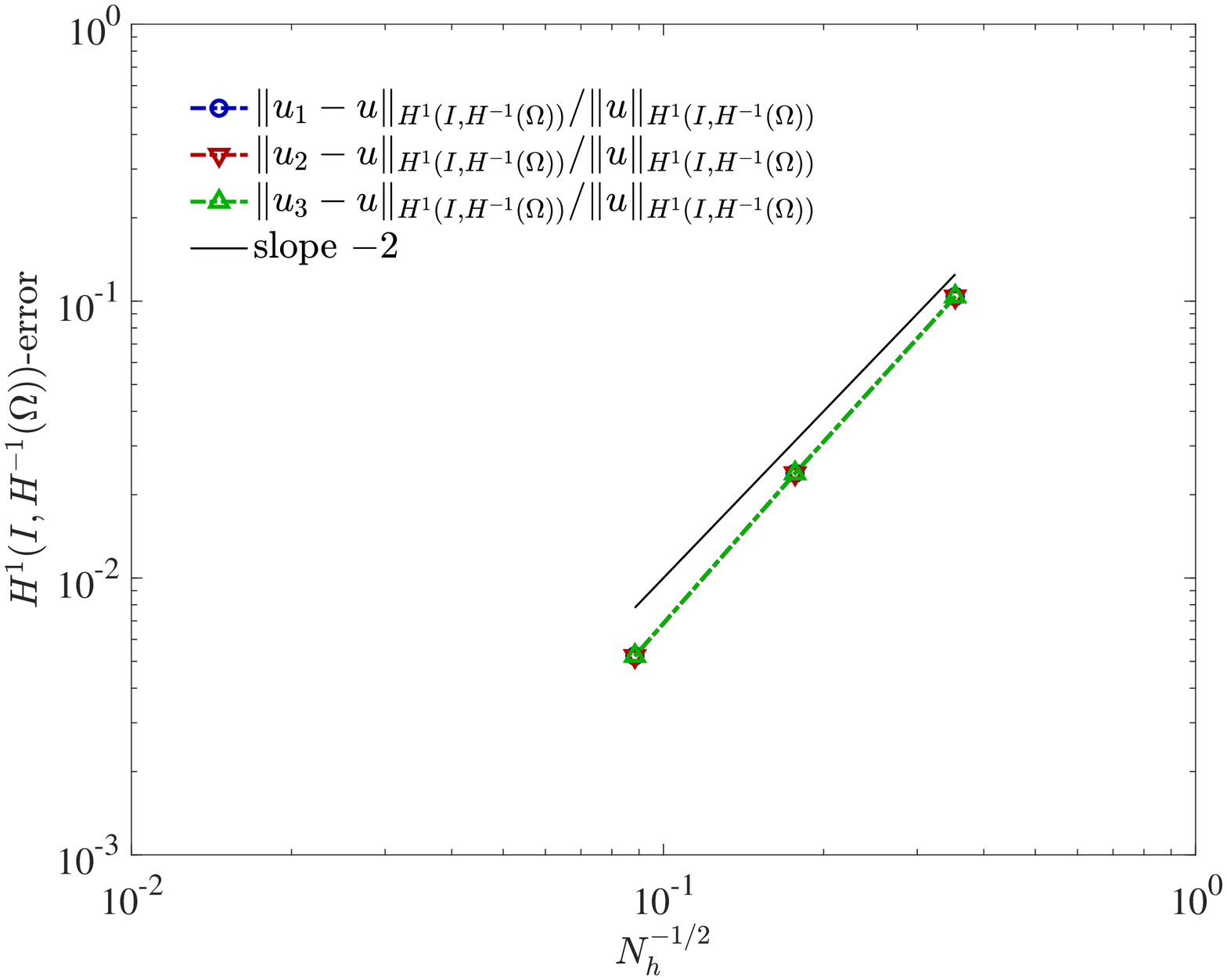} \qquad
\includegraphics[scale=0.2]{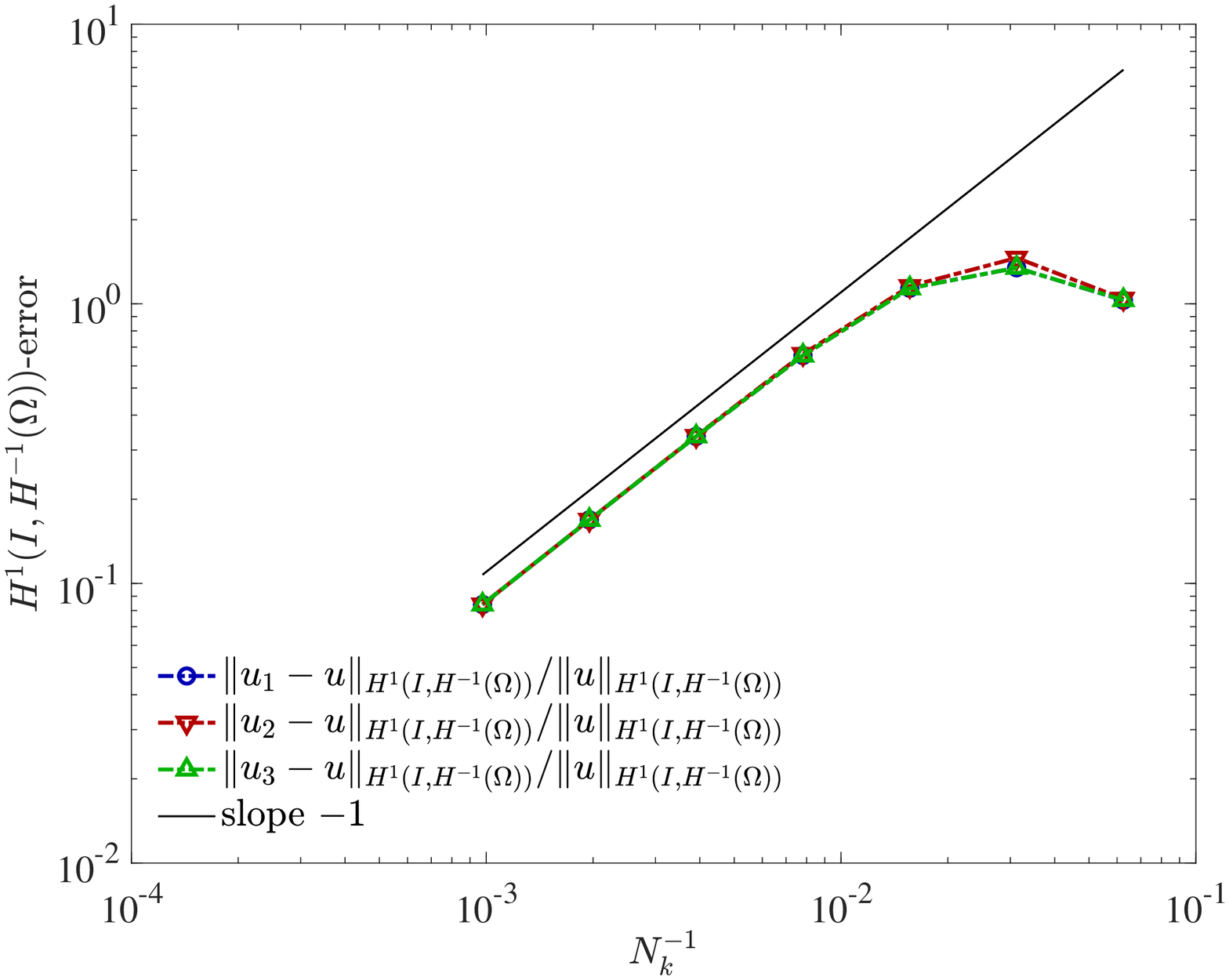}
\caption{Test case 2: convergence study for Methods 1, 2 and 3 for errors measured in the $L^2(I;H^1(\Omega))$-norm (top row) and in the $H^1(I;H^{-1}(\Omega))$-norm (bottom row) for various mesh-sizes $N_h^{-1/2}$ (left column) and various time-steps $N_k^{-1}$ (right column); in all cases, the curves for Methods~1 and~3 overlap.}
\label{fig:time_conv}
\end{center}
\end{figure}

Figure~\ref{fig:time_conv} presents a convergence study for Methods 1, 2 and 3 as a function of the discretization parameters $N_h$ (in space) and $N_k$ (in time). In both cases, we report the $L^2(I;H^1(\Omega))$- and $H^1(I;H^{-1}(\Omega))$-errors. The left panel considers $N_h=(2^l)^2$, $l\in\{2,3,4\}$ with $N_k=2^{13}$, whereas the right panel considers $N_k=2^l$, $l\in\{4,\ldots,10\}$ with $N_h=(2^6)^2$. Since the exact solution is not available, we consider for each method the approximate solution produced on the finest space-time discretization available. These method-dependent reference solutions are very close according to Figure~\ref{fig:time_err}, and their differences are, in both norms, two orders of magnitude lower than the convergence errors reported in Figure~\ref{fig:time_conv}. In this figure, we observe that for the three methods, the convergence rates are consistent with the best-approximation properties of the discrete trial space $X_{hk}$ in both norms: the convergence is of second order in time and first order in space in the $L^2(I;H^1(\Omega))$-norm, whereas it is of first order in time and second order in space in the $H^1(I;H^{-1}(\Omega))$-norm.

\subsection{Test case 3: advection-diffusion}

We consider in this section a time-independent, but non-selfadjoint, differential operator $A=-\nabla\cdot \mu\nabla +c(x,y)\cdot \nabla$ with diffusion coefficient $\mu=0.1$ and advection velocity field $c(x,y)=2\pi(\frac12-y,x-\frac12)^{\mathrm{T}}$. The source term is $f=0$
and the initial condition is $u_0(x,y)=\textup{exp}\Big(-\frac{(x-\frac23)^2+(y-\frac12)^2}{0.07^2}\Big)$. The explicit expression of the exact solution is not available. The discretization parameters are $N_h=(2^5)^2$ and $N_k=2^{10}$, and the stopping tolerances are $\epsilon_{\rm greedy}=10^{-5}$ and $\epsilon_{\rm alt}=5\times 10^{-2}$ (as in the previous test cases). The discretization parameters for this test case are a bit coarser than for the two other test cases because this test case turns out to be more computationally intensive. This is due to the fact that the differential operator in space is non-selfadjoint so that it is necessary to assemble the matrix $\bold{A}_h^{\rm T}\bold{D}_h^{-1}$ appearing in the definition~\eqref{eq:def_mathbb_B} of the global system matrix $\matB$ (recall that $P=1$ here and that $\bold{A}_h^{\mathrm{T}}=\bold{D}_h$ when the differential operator corresponds to a pure diffusion operator). 

\begin{figure}[ht]
\begin{center}
\includegraphics[scale=0.2]{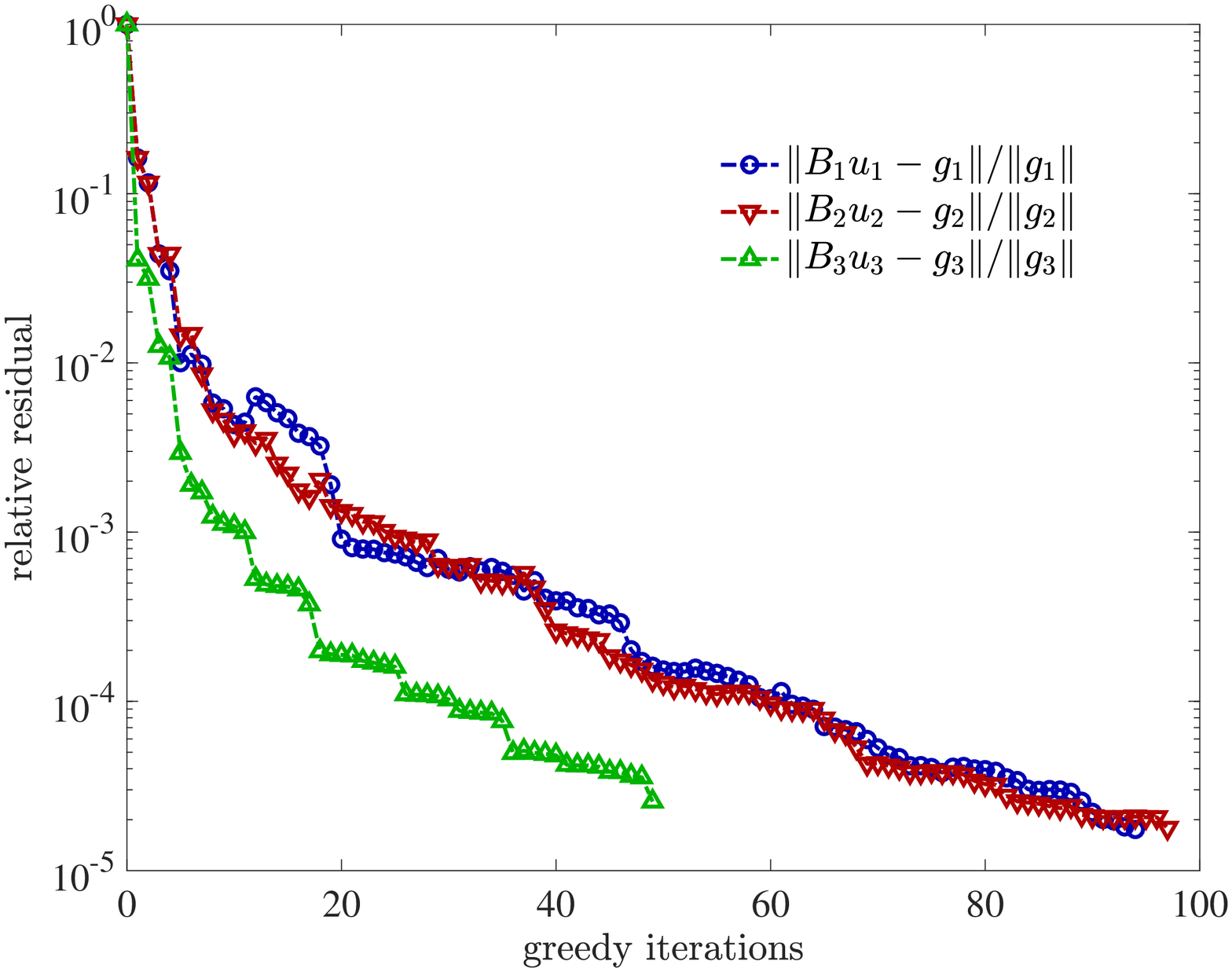} \qquad
\includegraphics[scale=0.2]{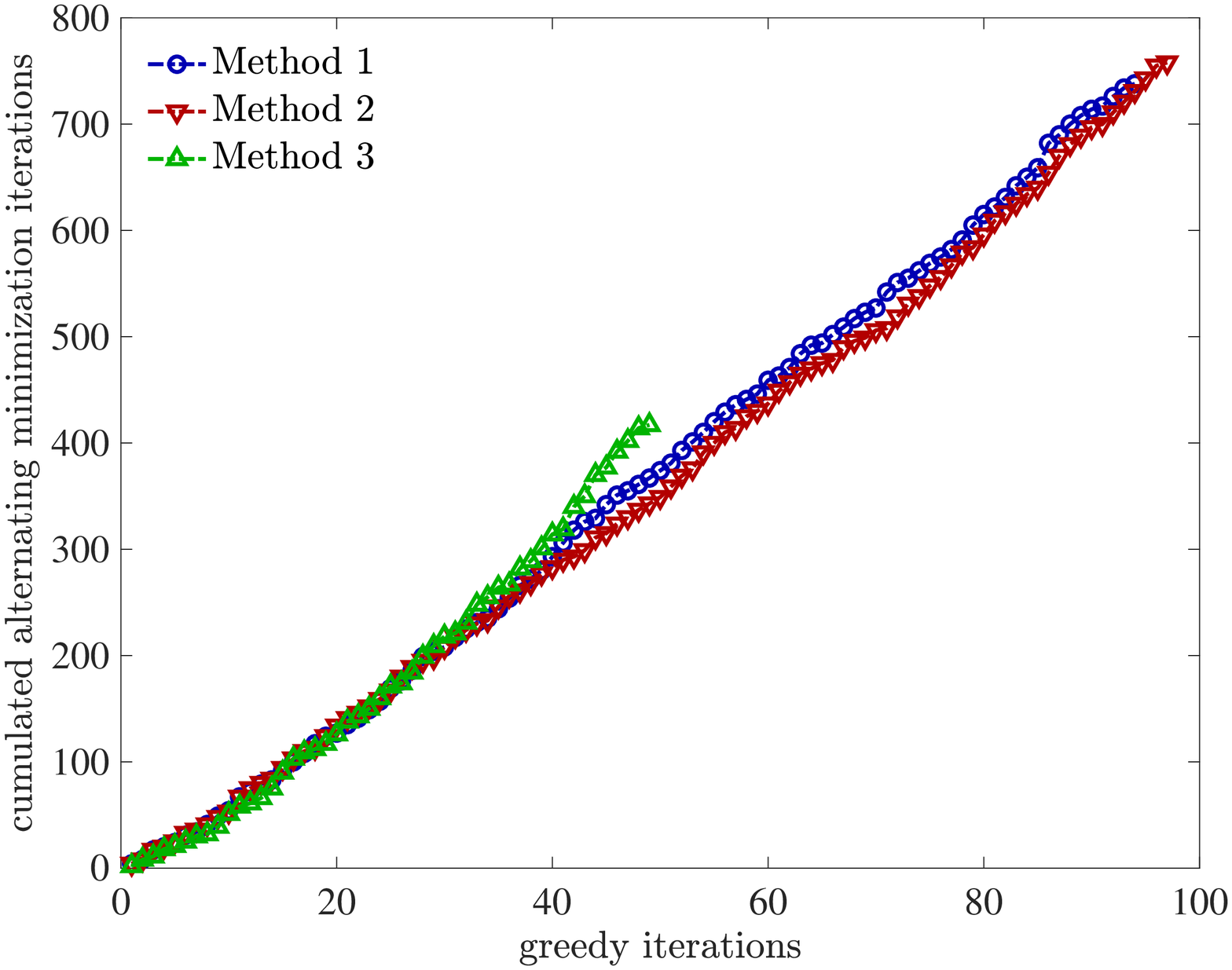}
\caption{Test case 3. Left: relative residual at each iteration of the greedy algorithm. Right: cumulated number of alternating minimization iterations in the greedy algorithm.}
\label{fig:ad_reserr}
\end{center}
\end{figure}

The left panel of Figure~\ref{fig:ad_reserr} presents the decrease of the relative residual as a function of the number of greedy iterations for Methods 1, 2 and 3. We notice that for Methods 1 and 2, the greedy algorithm takes around 90 iterations to converge (94 and 97, respectively), whereas it takes only 49 iterations for Method 3. Thus, for this test case, Method 3 takes less iterations. The right panel of Figure~\ref{fig:ad_reserr} presents the cumulated number of alternating minimization iterations in the greedy algorithm for Methods 1, 2 and 3. We observe that this number is about the same for the three methods. When reaching convergence for the greedy algorithm, we have solved about 750 linear systems in space, which is 73\% of the amount that would have been solved by using an implicit time-stepping method (recall that $N_k=2^{10}$). This percentage is larger than the ones reported for the previous two test cases, but is still competitive. 
Furthermore, similar observations as for test cases 1 and 2 can be made concerning the two contributions to the relative residual and the behavior of the residuals $r_i^m$ defined by~\eqref{eq:def_rim}. In particular, we have again 
$r_1^m\le \min(r_2^m,r_3^m)$ for all $m\ge0$ (as expected from the MinRes formulation); as the greedy algorithm approaches convergence, $r_1^m$ reaches a value of $2\times 10^{-5}$, whereas $r_2^m$ and $r_3^m$ reach a value of $7\times10^{-5}$ and $3\times 10^{-3}$, respectively.

\begin{figure}[ht]
\begin{center}
\includegraphics[scale=0.2]{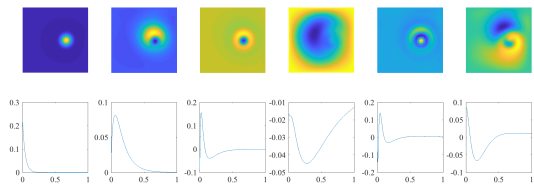}
\caption{Test case 3: first six modes in space (top row) and in time (bottom row) for Method 1.}
\label{fig:ad_modes}
\end{center}
\end{figure}

Figure~\ref{fig:ad_modes} presents the first six space and time modes for Method 1. The first six modes obtained with Method 2 are essentially the same, whereas some differences, especially in the space modes, can be observed with Method 3. This observation again confirms that the preconditioner plays a relevant role in the exploration of the discrete trial space.  

\begin{figure}[ht]
\begin{center}
\includegraphics[scale=0.2]{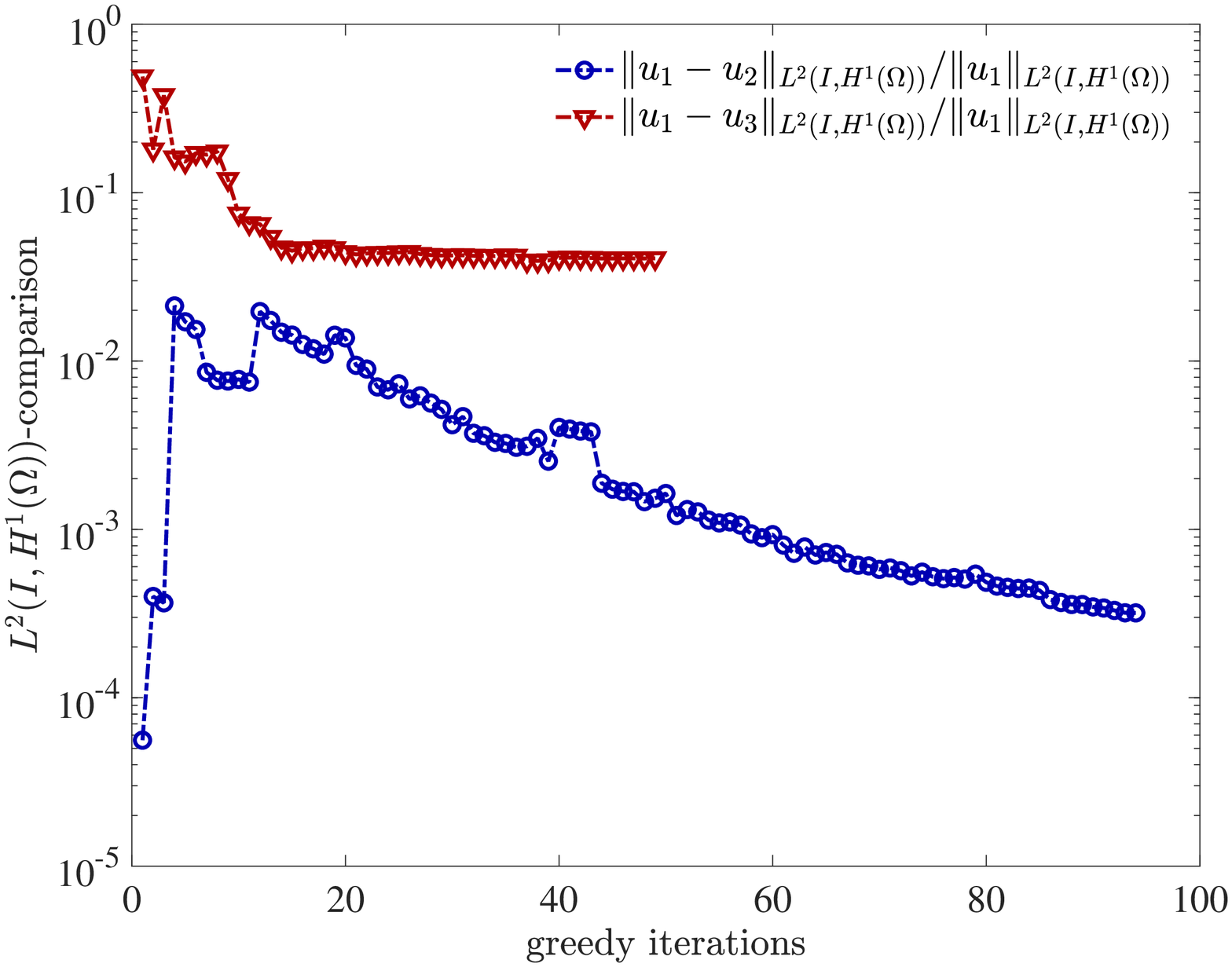} \qquad
\includegraphics[scale=0.2]{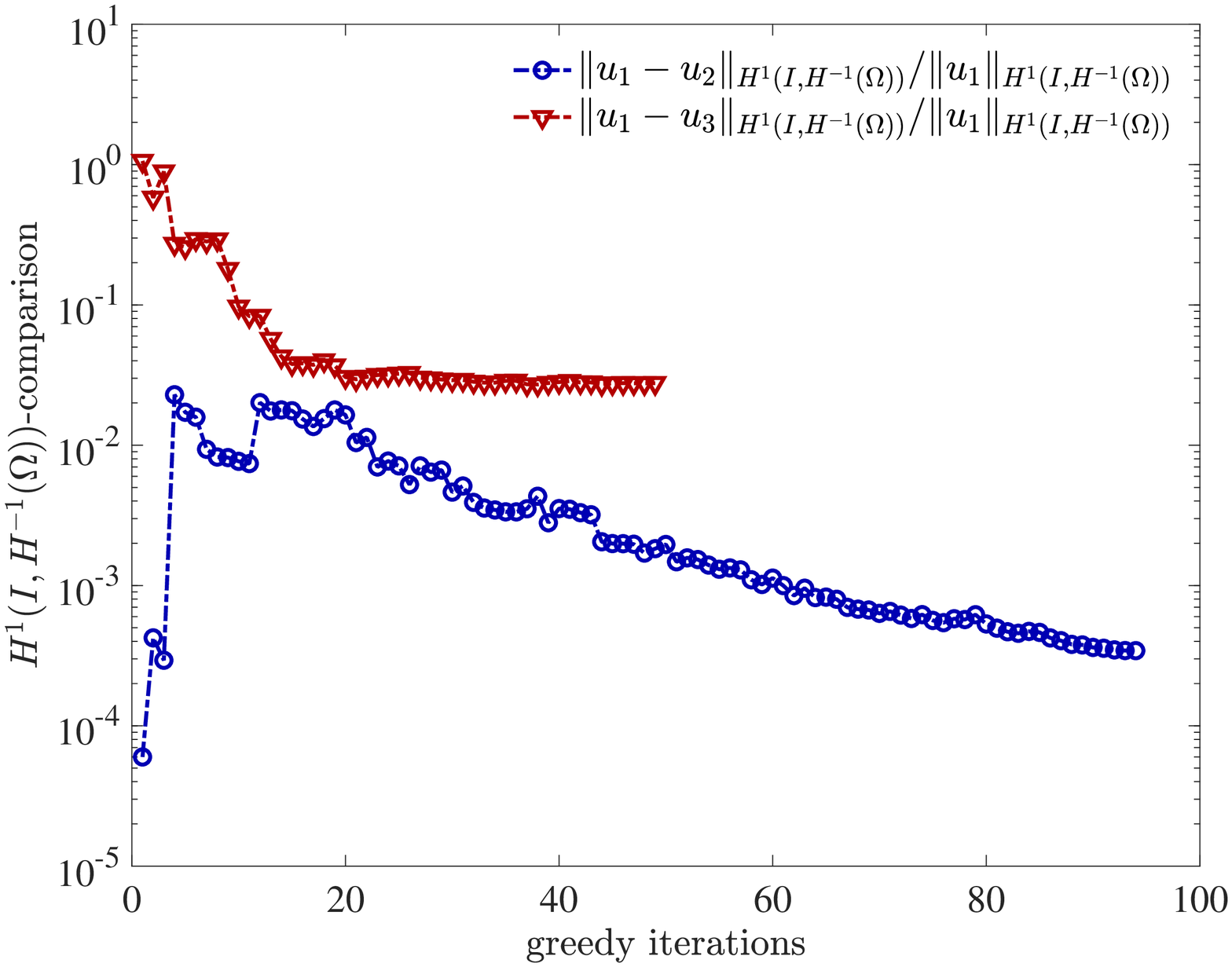} 
\caption{Test case 3: comparison of the solutions produced by Methods 1, 2, 3 in two norms: $L^2(I;H^1(\Omega))$ (left) and $H^1(I;H^{-1}(\Omega))$ (right).}
\label{fig:ad_err}
\end{center}
\end{figure}

Figure~\ref{fig:ad_err} reports the normalized differences $(u_{hk,1}^m-u_{hk,2}^m)$ and $(u_{hk,1}^m-u_{hk,3}^m)$ as a function of the iteration counter $m$ of the greedy algorithm where, as above, the additional subscript $i\in\{1,2,3\}$ indicates which method has been used. These differences are measured in the $L^2(I;H^1(\Omega))$- and $H^1(I;H^{-1}(\Omega))$-norms. We observe that the difference between the solutions produced by Methods 1 and 3 is significant in both norms (two orders of magnitude larger than the difference between Methods 1 and 2); therefore, we can conclude that Method 3 converges more rapidly than Methods 1 and 2, but with a poorer accuracy.

\begin{figure}[ht]
\begin{center}
\includegraphics[scale=0.2]{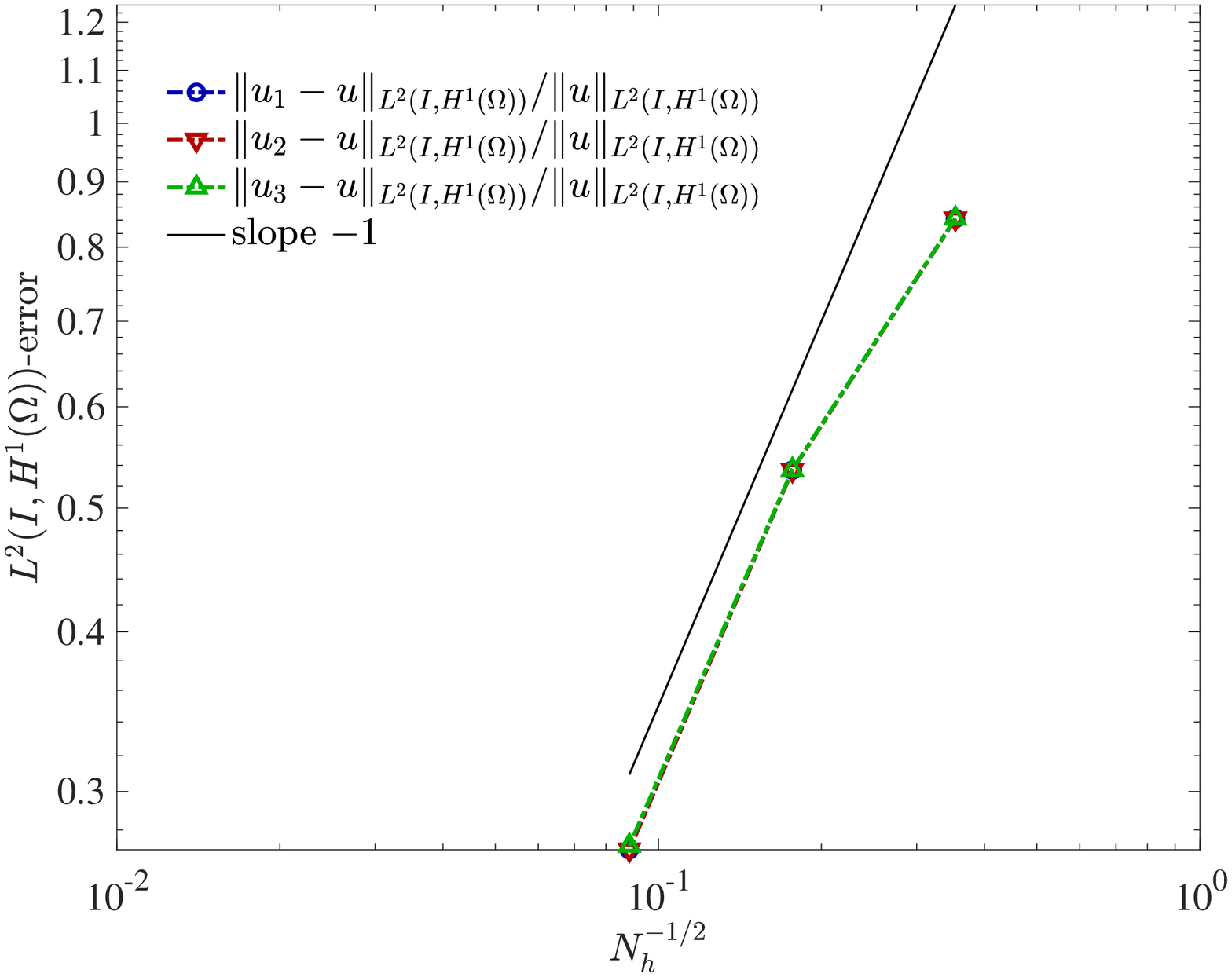} \qquad
\includegraphics[scale=0.2]{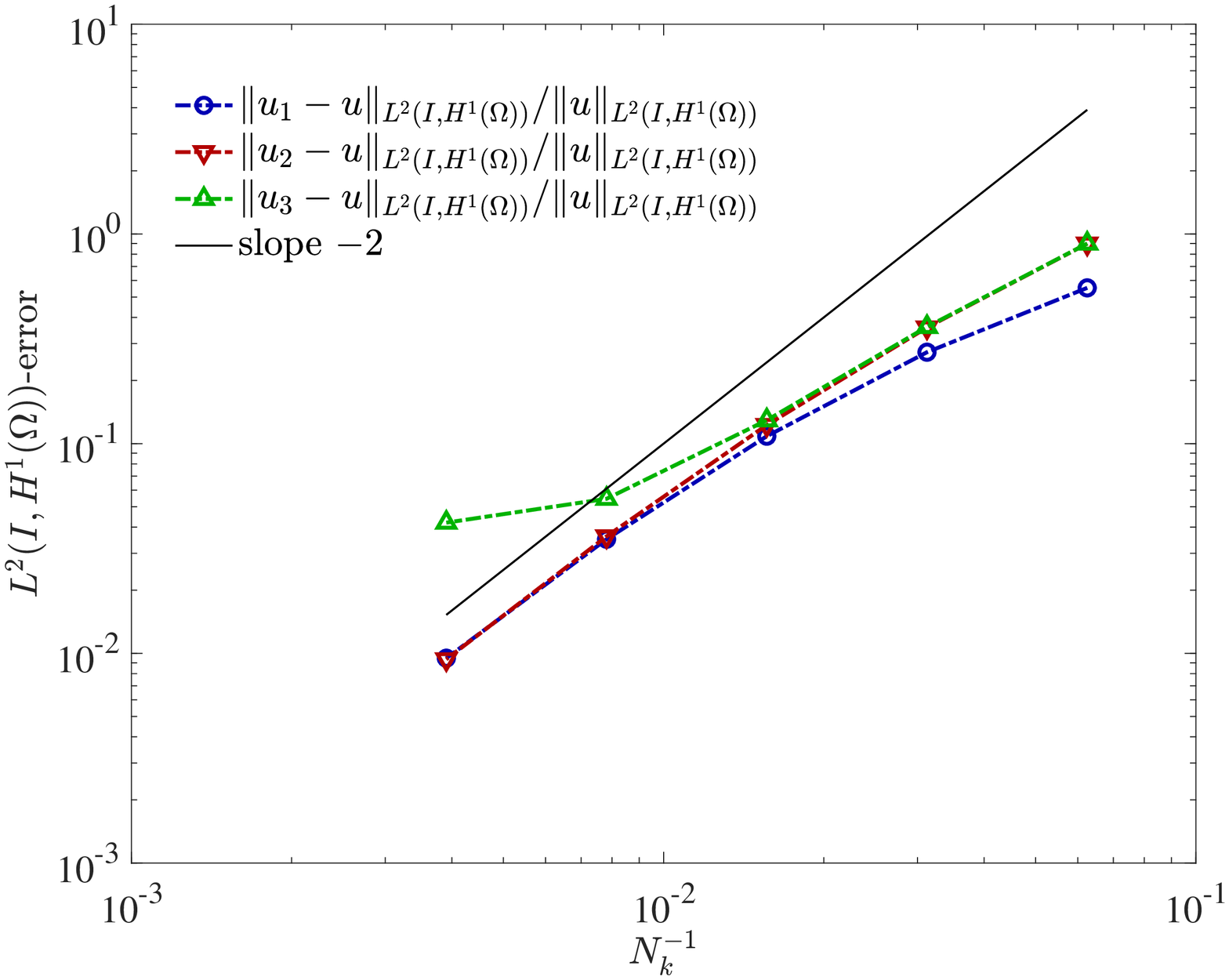} \\
\includegraphics[scale=0.2]{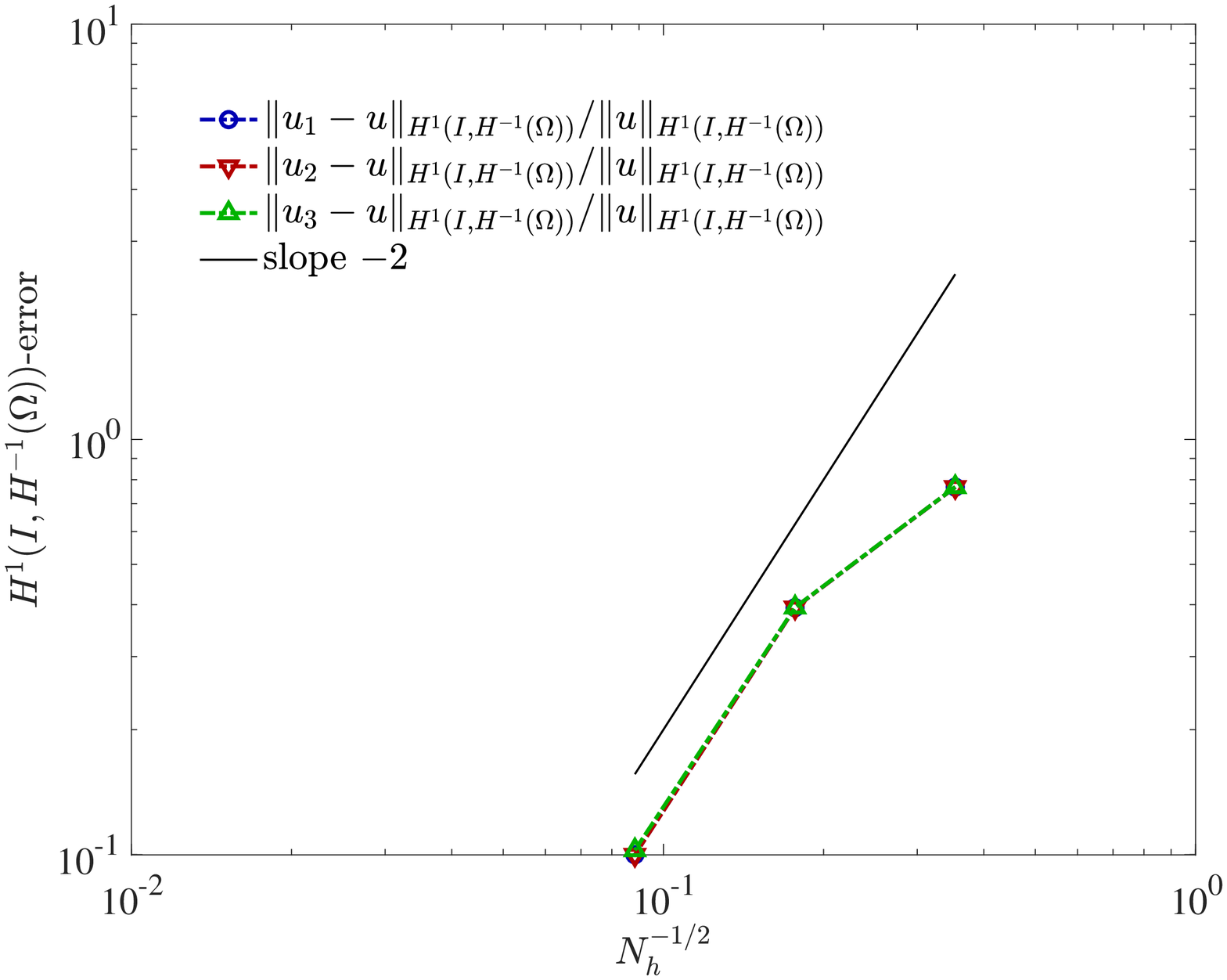} \qquad
\includegraphics[scale=0.2]{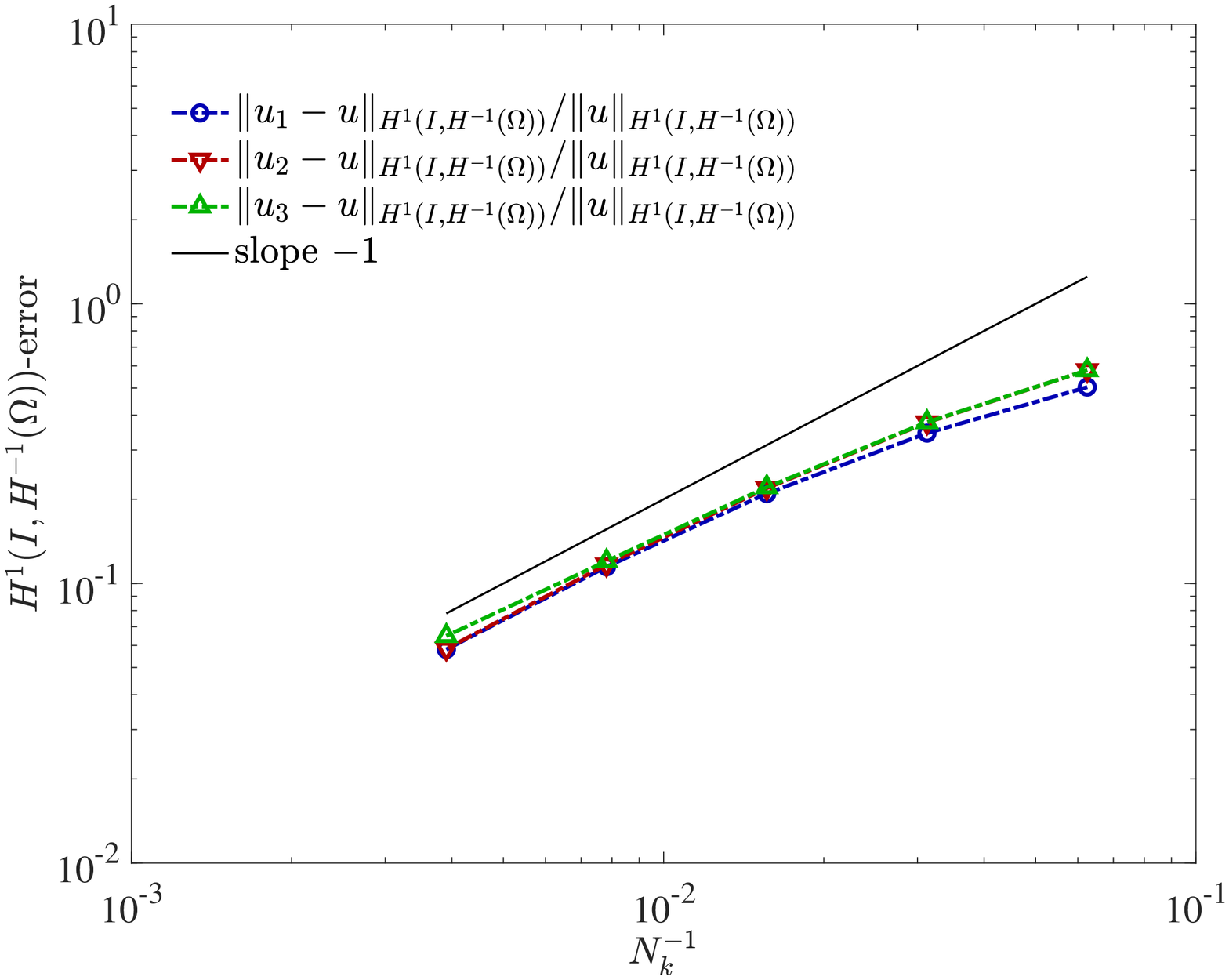}
\caption{Test case 3: convergence study for Methods 1, 2 and 3 for errors measured in the $L^2(I;H^1(\Omega))$-norm (top row) and in the $H^1(I;H^{-1}(\Omega))$-norm (bottom row) for various mesh-sizes $N_h^{-1/2}$ (left column) and various time-steps $N_k^{-1}$ (right column); in the right column, the curve for Method~2 overlaps with that of Method~1 for the smaller time-steps and with that of Method~3 for the larger time-steps.}
\label{fig:ad_conv}
\end{center}
\end{figure}

Figure~\ref{fig:ad_conv} presents a convergence study for Methods 1, 2 and 3 as a function of the discretization parameters $N_h$ (in space) and $N_k$ (in time). In both cases, we report the $L^2(I;H^1(\Omega))$- and $H^1(I;H^{-1}(\Omega))$-errors. The left panel considers $N_h=(2^{l})^2$, $l\in\{2,3,4\}$ with $N_k=2^{10}$, whereas the right panel considers $N_k=2^l$, $l\in\{4,\ldots,8\}$ with $N_h=(2^5)^2$. Since the exact solution is not available, we consider for each method the approximate solution produced on the finest space-time discretization available. For Methods 1 and 2, these two reference solutions are very close according to Figure~\ref{fig:ad_err}, and their difference is, in both norms, two orders of magnitude lower than the convergence errors reported in Figure~\ref{fig:ad_conv}. Moreover, for both methods, the reported convergence rates are, as above, consistent with the best-approximation properties of the discrete trial space $X_{hk}$ in both norms. Finally, for Method 3, the convergence rates are similar except for the behavior with respect to time refinement in the $L^2(I;H^1(\Omega))$-norm which is somewhat sub-optimal.

\section{Conclusion and outlook}
\label{sec:conclusion}

In this work, we have devised a space-time tensor for the low-rank approximation of linear parabolic evolution equations. The proposed method is uniformly stable with respect to the space-time discretization parameters and leads to solving sequentially global problems in space and in time. Our numerical results on various test cases indicate the importance of the preconditioner (which enters the method by means of the norm equipping the discrete test space), since the preconditioner has an influence on the convergence of the greedy algorithm. However, although the preconditioner improves the convergence of the greedy algorithm, it increases the cost of each iteration. Table~\ref{tab:cpu} summarizes the relative CPU times for Methods 2 and 3 normalized with respect to Method 1 (we have considered 21 random initializations in each case and used for each method the median CPU time). We can see from this table that Methods 1 and 2 essentially deliver the same CPU times, whereas Method 3 turns out to be more effective especially for test case 3 despite the increased number of greedy iterations.
Finally, various perspectives of this work can be envisaged. We mention in particular the question of adapting the discretization spaces and that of devising different approaches to obtain a separated representation of the exact solution with sufficient accuracy at low-rank when the differential operator has a dominant non-selfadjoint part, as in advection-dominated transport problems.

\begin{table}
\begin{center}\begin{tabular}{|l|ccc|}
\hline
Test case&1&2&3\\
\hline
Method 2&0.92&1.01&1.02\\
Method 3&1.17&0.82&0.38\\
\hline
\end{tabular}
\end{center}
\caption{Relative CPU times for Methods 2 and 3 with respect to Method 1 for the three test cases.}
\label{tab:cpu}
\end{table}

\bibliographystyle{acm}
\bibliography{Bibliography}

\begin{thebibliography}{10}

\bibitem{Andreev:12}
{\sc Andreev, R.}
\newblock {\em Stability of space-time {P}etrov--{G}alerkin discretizations for
  parabolic evolution equations}.
\newblock PhD thesis, ETH Z\"urich, 2012.

\bibitem{Andreev_13}
{\sc Andreev, R.}
\newblock Stability of sparse space-time finite element discretizations of
  linear parabolic evolution equations.
\newblock {\em IMA J. Numer. Anal. 33}, 1 (2013), 242--260.

\bibitem{Andreev_2014_a}
{\sc Andreev, R.}
\newblock Space-time discretization of the heat equation.
\newblock {\em Numer. Algorithms 67}, 4 (2014), 713--731.

\bibitem{BaMNP:04}
{\sc Barrault, M., Maday, Y., Nguyen, N.~C., and Patera, A.~T.}
\newblock An `empirical interpolation' method: application to efficient
  reduced-basis discretization of partial differential equations.
\newblock {\em C. R. Math. Acad. Sci. Paris 339}, 9 (2004), 667--672.

\bibitem{Cances_2011_a}
{\sc Canc{{\`e}}s, E., Ehrlacher, V., and Leli{{\`e}}vre, T.}
\newblock Convergence of a greedy algorithm for high-dimensional convex
  nonlinear problems.
\newblock {\em Math. Models Methods Appl. Sci. 21\/} (2011), 2433--2467.

\bibitem{cances2014greedy}
{\sc Canc{\`e}s, E., Ehrlacher, V., and Leli{\`e}vre, T.}
\newblock Greedy algorithms for high-dimensional eigenvalue problems.
\newblock {\em Constructive Approximation 40}, 3 (2014), 387--423.

\bibitem{Chinesta_2014_a}
{\sc Chinesta, F., Keunings, R., and Leygue, A.}
\newblock {\em The Proper Generalized Decomposition for Advanced Numerical
  Simulations}.
\newblock Springer Briefs in Applied Sciences and Technology. Springer, Cham,
  2014.
\newblock A primer.

\bibitem{DauLi:92}
{\sc Dautray, R., and Lions, J.-L.}
\newblock {\em Mathematical Analysis and Numerical Methods for Science and
  Technology. {V}ol. 5. Evolution problems, I.}
\newblock Springer-Verlag, Berlin, Germany, 1992.

\bibitem{Ern_2004_a}
{\sc Ern, A., and Guermond, J.-L.}
\newblock {\em Theory and Practice of Finite Elements}, vol.~159 of {\em
  Applied Mathematical Sciences}.
\newblock Springer-Verlag, New York, 2004.

\bibitem{ErnSV:17}
{\sc Ern, A., Smears, I., and Vohral{\'\i}k, M.}
\newblock Guaranteed, {L}ocally {S}pace-{T}ime {E}fficient, and
  {P}olynomial-{D}egree {R}obust a {P}osteriori {E}rror {E}stimates for
  {H}igh-{O}rder {D}iscretizations of {P}arabolic {P}roblems.
\newblock {\em SIAM J. Numer. Anal. 55}, 6 (2017), 2811--2834.

\bibitem{Falco_2011_a}
{\sc Falc{{\'o}}, A., and Nouy, A.}
\newblock A proper generalized decomposition for the solution of elliptic
  problems in abstract form by using a functional {E}ckart-{Y}oung approach.
\newblock {\em J. Math. Anal. Appl. 376}, 2 (2011), 469--480.

\bibitem{Falco_2012_a}
{\sc Falc{{\'o}}, A., and Nouy, A.}
\newblock Proper generalized decomposition for nonlinear convex problems in
  tensor {B}anach spaces.
\newblock {\em Numer. Math. 121}, 3 (2012), 503--530.

\bibitem{GanVa:07}
{\sc Gander, M.~J., and Vandewalle, S.}
\newblock Analysis of the parareal time-parallel time-integration method.
\newblock {\em SIAM J. Sci. Comput. 29}, 2 (2007), 556--578.

\bibitem{GanZh:02}
{\sc Gander, M.~J., and Zhao, H.}
\newblock Overlapping {S}chwarz waveform relaxation for the heat equation in
  {$n$} dimensions.
\newblock {\em BIT 42}, 4 (2002), 779--795.

\bibitem{GilKe:02}
{\sc Giladi, E., and Keller, H.~B.}
\newblock Space-time domain decomposition for parabolic problems.
\newblock {\em Numer. Math. 93}, 2 (2002), 279--313.

\bibitem{GriOe:07}
{\sc Griebel, M., and Oeltz, D.}
\newblock A sparse grid space-time discretization scheme for parabolic
  problems.
\newblock {\em Computing 81}, 1 (2007), 1--34.

\bibitem{GunKu:11}
{\sc Gunzburger, M.~D., and Kunoth, A.}
\newblock Space-time adaptive wavelet methods for optimal control problems
  constrained by parabolic evolution equations.
\newblock {\em SIAM J. Control Optim. 49}, 3 (2011), 1150--1170.

\bibitem{HoaSc:13}
{\sc Hoang, V.~H., and Schwab, C.}
\newblock Sparse tensor {G}alerkin discretization of parametric and random
  parabolic {PDE}s---analytic regularity and generalized polynomial chaos
  approximation.
\newblock {\em SIAM J. Math. Anal. 45}, 5 (2013), 3050--3083.

\bibitem{JanVa:96}
{\sc Janssen, J., and Vandewalle, S.}
\newblock Multigrid waveform relaxation of spatial finite element meshes: the
  continuous-time case.
\newblock {\em SIAM J. Numer. Anal. 33}, 2 (1996), 456--474.

\bibitem{KiLuW:16}
{\sc Kieri, E., Lubich, C., and Walach, H.}
\newblock Discretized dynamical low-rank approximation in the presence of small
  singular values.
\newblock {\em SIAM J. Numer. Anal. 54}, 2 (2016), 1020--1038.

\bibitem{KocLu:07}
{\sc Koch, O., and Lubich, C.}
\newblock Dynamical low-rank approximation.
\newblock {\em SIAM J. Matrix Anal. Appl. 29}, 2 (2007), 434--454.

\bibitem{Ladeveze_2012_a}
{\sc Ladev{\`e}ze, P.}
\newblock {Nonlinear Computational Structural Mechanics: New Approaches and
  Non-Incremental Methods of Calculation}, 2012.

\bibitem{Le-Bris_2009_a}
{\sc Le~Bris, C., Leli{{\`e}}vre, T., and Maday, Y.}
\newblock Results and questions on a nonlinear approximation approach for
  solving high-dimensional partial differential equations.
\newblock {\em Constr. Approx. 30}, 3 (2009), 621--651.

\bibitem{LiMaT:01}
{\sc Lions, J.-L., Maday, Y., and Turinici, G.}
\newblock R\'esolution d'{EDP} par un sch\'ema en temps ``parar\'eel''.
\newblock {\em C. R. Acad. Sci. Paris S\'er. I Math. 332}, 7 (2001), 661--668.

\bibitem{LioMa:72}
{\sc Lions, J.-L., and Magenes, E.}
\newblock {\em Non-homogeneous boundary value problems and applications.
  {V}ols. {I}, {II}}.
\newblock Springer-Verlag, New York-Heidelberg, 1972.
\newblock Translated from the French by P. Kenneth, Die Grundlehren der
  mathematischen Wissenschaften, Band 181-182.

\bibitem{LubOs:14}
{\sc Lubich, C., and Oseledets, I.~V.}
\newblock A projector-splitting integrator for dynamical low-rank
  approximation.
\newblock {\em BIT 54}, 1 (2014), 171--188.

\bibitem{MaScTo:18}
{\sc Mantzaflaris, A., Scholz, F., and Toulopoulos, I.}
\newblock Low-rank space-time isogeometric analysis for parabolic problems with
  varying coefficients.
\newblock {\em Comp. Methods Appl. Math.\/} (2018).
\newblock Published online, DOI
  \texttt{https://doi.org/10.1515/cmam-2018-0024}.

\bibitem{NeuSme:18}
{\sc Neum\"uller, M., and Smears, I.}
\newblock Time-parallel iterative solvers for parabolic evolution equations.
\newblock arXiv:1802.08126, 2018.

\bibitem{Nouy_2009_a}
{\sc Nouy, A.}
\newblock Recent developments in spectral stochastic methods for the numerical
  solution of stochastic partial differential equations.
\newblock {\em Arch. Comput. Methods Eng. 16}, 3 (2009), 251--285.

\bibitem{Nouy_2010_a}
{\sc Nouy, A.}
\newblock {A priori model reduction through Proper Generalized Decomposition
  for solving time-dependent partial differential equations}.
\newblock {\em Comput. Methods Appl. Mech. Engrg. 199}, 23-24 (2010),
  1603--1626.

\bibitem{PaiSa:82}
{\sc Paige, C.~C., and Saunders, M.~A.}
\newblock L{SQR}: an algorithm for sparse linear equations and sparse least
  squares.
\newblock {\em ACM Trans. Math. Software 8}, 1 (1982), 43--71.

\bibitem{Schwab_2009_a}
{\sc Schwab, C., and Stevenson, R.}
\newblock Space-time adaptive wavelet methods for parabolic evolution problems.
\newblock {\em Math. Comp. 78}, 267 (2009), 1293--1318.

\bibitem{Tantardini_2016_a}
{\sc Tantardini, F., and Veeser, A.}
\newblock The {$L^2$}-projection and quasi-optimality of {G}alerkin methods for
  parabolic equations.
\newblock {\em SIAM J. Numer. Anal. 54}, 1 (2016), 317--340.

\bibitem{Temlyakov_2008_a}
{\sc Temlyakov, V.~N.}
\newblock Greedy approximation.
\newblock {\em Acta Numer. 17\/} (2008), 235--409.

\bibitem{Thomee:06}
{\sc Thom\'ee, V.}
\newblock {\em Galerkin finite element methods for parabolic problems},
  second~ed., vol.~25 of {\em Springer Series in Computational Mathematics}.
\newblock Springer-Verlag, Berlin, 2006.

\bibitem{UrbPa:12}
{\sc Urban, K., and Patera, A.~T.}
\newblock A new error bound for reduced basis approximation of parabolic
  partial differential equations.
\newblock {\em C. R. Math. Acad. Sci. Paris 350\/} (2012).

\bibitem{Uschmajew:12}
{\sc Uschmajew, A.}
\newblock Local convergence of the alternating least squares algorithm for
  canonical tensor approximation.
\newblock {\em SIAM J. Matrix Anal. Appl. 33}, 2 (2012), 639--652.

\bibitem{Wloka:87}
{\sc Wloka, J.}
\newblock {\em Partial differential equations}.
\newblock Cambridge University Press, Cambridge, 1987.
\newblock Translated from the German by C. B. Thomas and M. J. Thomas.

\end{thebibliography}
\end{document}